\documentclass[oneside]{amsart}
    \usepackage{amsmath}
    \usepackage{amsfonts}
    \usepackage{amsthm}
    \usepackage{amssymb}
    \usepackage{hyperref}
    \usepackage{amscd}
    \usepackage{color}
    \usepackage{enumerate}
    \usepackage{xypic}
    \usepackage{tikz}
    \usepackage{tikz-cd}
    \usetikzlibrary{arrows}
    \usetikzlibrary{patterns}
    \usepackage{colortbl}
    \usepackage[normalem]{ulem}
    \usepackage{multirow}
    \usepackage{makecell}

    \usepackage{booktabs}

    \usepackage{dsfont}
    
    \makeatletter
    \def\l@subsection{\@tocline{2}{0pt}{2.5pc}{5pc}{}}
    \makeatother

    \newcommand{\M}{\mathcal{M}}
    \newcommand{\A}{\mathcal{A}}
    \newcommand{\B}{\mathcal{B}}
    \newcommand{\ZZ}{\mathbb{Z}}
    \newcommand{\KK}{\mathbb{K}}
    \newcommand{\RR}{\mathbb{R}}
    \newcommand{\PP}{\mathbb{P}}
    
    \newcommand{\QQ}{\mathbb{Q}}
    \newcommand{\GG}{\mathbb{G}}
    \newcommand{\NN}{\mathbb{N}}
    \newcommand{\ff}{\mathbf{f}}
    \newcommand{\bp}{\mathbf{p}}

    \newcommand{\Hilb}{\mathrm{Hilb}}
    \newcommand{\Chow}{\mathrm{Chow}}

      \DeclareMathOperator{\Ch}{Ch}
    \DeclareMathOperator{\mult}{mult}
    \DeclareMathOperator{\Res}{Res}
    \DeclareMathOperator{\Red}{Red}
    \DeclareMathOperator{\Char}{char}
    
    \DeclareMathOperator{\spec}{Spec}
    \DeclareMathOperator{\Proj}{Proj}
    
    \DeclareMathOperator{\conv}{conv}
    \DeclareMathOperator{\rank}{rank}

    \DeclareMathOperator{\Gr}{Gr}
    
    \DeclareMathOperator{\Hom}{Hom}
    
    \DeclareMathOperator{\Stab}{Stab}
    \newcommand{\Aff}{\mathbb{A}}
    \newcounter{myintro}

    \newtheorem{thm*}[myintro]{Theorem}
    \newtheorem{thm}[equation]{Theorem}
    \newtheorem{prop}[equation]{Proposition}
    \newtheorem{lemma}[equation]{Lemma}
    \newtheorem{cor}[equation]{Corollary}

    \theoremstyle{definition}
    
    \newtheorem{defn}[equation]{Definition}

    \newtheorem{rem}[equation]{Remark}

    \newtheorem{question}[equation]{Question}
    \newtheorem{oldex}[equation]{Example}
    \newtheorem{oldexB}[myintro]{Example}
    
    \newenvironment{ex}
      {\pushQED{\qed}
      \oldex}
      {\popQED\endoldex}
    
    \newenvironment{introex}
      {\pushQED{\qed}
      \oldexB}
      {\popQED\endoldexB}

    \newcommand{\hide}[1]{}
    
    \numberwithin{equation}{subsection}
    \numberwithin{myintro}{section}

    \title{Rational Curves in Projective Toric Varieties}
    
    \author{Nathan Ilten}
    \address{Department of Mathematics, Simon Fraser University,
    8888 University Drive, Burnaby BC V5A1S6, Canada}
    \email{\href{mailto:nilten@sfu.ca}{nilten@sfu.ca}}
    
    \author{Jake Levinson}
\address{Département de mathématiques et de statistique, Université de Montréal, Montréal, QC
Canada H3C 3J7}
    \email{\href{mailto:jake.levinson@umontreal.ca}{jake.levinson@umontreal.ca}}

\begin{document}
    
    \begin{abstract}
	    We study embedded rational curves in projective toric varieties. Generalizing results of the first author and Zotine for the case of lines, we show that any degree $d$ rational curve in a toric variety $X$ can be constructed from a special affine-linear map called a degree $d$ Cayley structure. We characterize when the curves coming from a degree $d$ Cayley structure are smooth and have degree $d$.  We use this to establish a bijection between the set of irreducible components of the Hilbert scheme whose general element is a smooth degree $d$ curve, and so-called maximal smooth Cayley structures. 
Furthermore, we describe the normalization of the torus orbit closure of such rational curves in the Chow variety, and give partial results for the orbit closures in the Hilbert scheme.
    \end{abstract}
    \maketitle

    \tableofcontents
    \section{Introduction}
    An insightful approach to better understanding the geometry of a projective variety $X\subseteq \PP^n$ involves studying all subvarieties contained in $X$. Of particular interest are rational curves. In this paper, we will contribute to the study of rational curves in projective toric varieties. 
 
    Rational curves are often studied by considering (log stable) maps $\PP^1 \to X$, see \cite{stablemaps} and \cite{banarjee} for results in the toric setting. However, our focus will not be on the map, but on the curve itself, that is, on the image of the map. In particular, the underlying moduli spaces we are primarily interested in are the Hilbert schemes $\Hilb_{dm+1}(X)$ (parametrizing subschemes of $X$ with Hilbert polynomial $P(m)=dm+1$ equal to that of a smooth rational curve of degree $d$) and $\Chow_d(X)$ (parametrizing one-cycles of $X$ of degree $d$).  
This paper builds on previous work of the first author and Zotine \cite{fano}, in which they study Fano schemes of toric varieties.
The special case of the Fano scheme of lines is that of rational curves of degree $d=1$, in which case the Chow and Hilbert schemes coincide. 

    We now summarize our main results. For this purpose, we first introduce a bit of notation. Throughout, we will work over an algebraically closed field $\KK$ of characteristic zero. Let $\A$ be a finite subset of a lattice $M\cong \ZZ^n$. Associated to $\A$ is a (not necessarily normal) projective toric variety $X_\A\subseteq \PP^{\#\A-1}$, see \S\ref{sec:toric}.
    The key combinatorial gadgets in our study of $\Hilb_{dm+1}(X_\A)$ are \emph{degree $d$ Cayley structures of length $\ell$} on faces of the set $\A$, see \S \ref{sec:cayley} for a definition. Roughly speaking, a degree $d$ Cayley structure of length $\ell$ is an affine linear map mapping a subset $\tau$ of $\A$ to the $d$th dilate of a standard simplex of dimension $\ell$. Any two Cayley structures are equivalent if they differ by a permutation of the vertices of the simplex.

    A degree $d$ Cayley structure $\pi$ defines a family over $M_{0,\ell+1}\times T_\tau$ of non-constant basepoint-free maps from $\PP^1$ to $X_\A$ (see \S\ref{sec:curveconst}). Here, $T_\tau$ is a quotient of the dense torus of $X_\A$, and $M_{0,\ell+1}$ the moduli space of $\ell+1$-marked points on $\PP^1$. (In the special case $\ell=1$, the family is over $T_\tau/\KK^*$).
    We identify a combinatorial criterion on the Cayley structure $\pi$ that characterizes when the image of a generic map in this family has degree $d$ (corresponding to $\pi$ being \emph{primitive}, see Definition \ref{defn:primitive} and Theorem \ref{thm:primitive}). Likewise, we identify a combinatorial criterion that characterizes when the image of a generic map in the family is smooth (corresponding to $\pi$ being \emph{smooth}, see Definition \ref{defn:smooth} and Theorem \ref{thm:smooth}).
    For a smooth primitive degree $d$ Cayley structure $\pi$, we thus have a rational map 
    \[
   M_{0,\ell+1}\times T_\tau \dashrightarrow \Hilb_{dm+1}.\]
    In fact, this map is generically finite. We denote the closure of its image by $Z_\pi$.

    There is a natural combinatorially-defined partial order $\leq$ on the set of smooth primitive Cayley structures of degree $d$ defined on faces of $\A$ (see Definition \ref{defn:order}). Using this, we obtain: 
    \begin{thm*}[See Corollary \ref{cor:main}]
	    The map $\pi \mapsto Z_\pi$ induces a bijection between equivalence classes of maximal smooth primitive degree $d$ Cayley structures and irreducible components of $\Hilb_{dm+1}(X_\A)$ whose general element is a smooth rational curve.
    \end{thm*}

    The most interesting behaviour of many moduli spaces is found along the boundary. Motivated by this, we study the limiting behaviour of a general element $\eta$ of $Z_\pi$ under a one-parameter subgoup of $T_\tau$. In Theorem \ref{thm:limit}, we give a combinatorial description of this limit as a one-cycle. Using this we are able to describe the normalization of the closure of the $T_\tau$-orbit of $\eta$ in the Chow variety $\Chow_d(X_\A)$. Indeed, using the combinatorics of $\pi$, we construct a fan $\Sigma_\pi$ (see Definition \ref{defn:sigmapi}) and prove the following:
    \begin{thm*}[See Theorem \ref{thm:fan}]
The normalization of the $T_\tau$-orbit closure of a general curve corresponding to $\pi$ is the toric variety corresponding to the fan $\Sigma_\pi$.
    \end{thm*}

    Understanding the normalization of the $T_\tau$-orbit closure in $\Hilb_{dm+1}$ of a general point $\eta\in Z_\pi$ is much more subtle. Using the Hilbert-Chow morphism, we see that this toric variety is described by a fan given by a refinement of $\Sigma_\pi$ (see Proposition \ref{prop:hilb}). However, in the case of the Hilbert scheme of conics, we are able to say exactly what happens: the fan is the coarsest common refinement of $\Sigma_\pi$ with the normal fan $\Sigma'$ of a certain matroid polytope (see Theorem \ref{thm:conics}).

   The combinatorics of the Cayley structures we consider are essential for guaranteeing that the corresponding curves are contained in a given toric variety. Nonetheless, many of the arguments in this paper can be reduced to the case of rational curves in projective space. A number of general results are known  for rational curves in $\PP^n$ (see e.g. \cite{cubics} on rational cubic curves and \cite{normalb} on decompositions by the splitting type of the normal bundle). 
   In our setting, however, we are forced to consider families of rational curves arising from parametrizations given by products of linear forms in two variables. To our knowledge, results on such families, like our characterization of when the general curve is  non-singular (Theorem \ref{thm:smooth}), cannot be found in the existing literature.

    We briefly comment on the relationship between this paper and \cite{stablemaps,banarjee}. In \cite{stablemaps}, Ranganathan considers the moduli space of log stable maps from $\PP^1$ to a toric variety $X$. The interior of this moduli space, corresponding to maps from $\PP^1$ whose images meet the dense torus of $X$ and intersect the boundary in prescribed fashion, is similar to the family of maps we obtain from a Cayley structure $\pi$. However, the behaviour at the boundary is very different from what happens in the Hilbert or Chow schemes. In \cite{banarjee}, Banerjee gives a combinatorial description of the space of morphisms of fixed multidegree from $\PP^1$ to a simplicial toric variety $X$. This is perhaps more similar to our approach, in the sense that such maps can have images that are contained in the toric boundary. Nonetheless, the focus there remains on the morphism, not the image as in our case.

    The remainder of this paper is organized as follows. In \S\ref{sec:prelim} we introduce some basic notation for toric varieties and define degree $d$ Cayley structures and related notions. We show in \S\ref{sec:curves} how to construct a family of rational curves from any Cayley structure, and conversely how a rational curve in a toric variety determines a corresponding Cayley structure.
    In \S\ref{sec:nodesandcusps} we more closely study the geometry of the rational curves obtained from a degree $d$ Cayley structure, characterizing in particular when these curves are smooth and of degree $d$. 
    Our discussion on irreducible components of the Hilbert scheme is found in \S\ref{sec:hilb}. We conclude in \S\ref{sec:orbits} with a study of torus orbits in the Chow and Hilbert schemes.
We finish this introduction with an example illustrating some of our results.
\begin{introex}[A singular Fano threefold]\label{ex:fano}
Let $\A$ be the subset of $\RR^3$ whose elements are the columns of the matrix 
\[
\left(	\begin{array}{c c c c c c c c c}
-1&0&1&1&0&-1&0&0&0\\
-1&-1&0&1&1&0&0&0&0\\
0&0&0&0&0&0&0&1&-1\\
	\end{array}\right),
\]
see Figure \ref{fig:introex}. The toric variety $X_\A$ is a projectively normal singular Fano threefold in $\PP^8$. The set $\A$ admits $9$ non-equivalent maximal degree two Cayley structures of length $1$:
\begin{align*}
	(u_1,u_2,u_3)&\mapsto (1+u_1+c\cdot u_3,1-u_1-c\cdot u_3)\qquad\qquad c\in \{-1,0,1\}\\
	(u_1,u_2,u_3)&\mapsto (1+u_2+c\cdot u_3,1-u_2-c\cdot u_3)\qquad\qquad c\in \{-1,0,1\}\\
	(u_1,u_2,u_3)&\mapsto (1+u_1-u_2+c\cdot u_3,1-u_1+u_2-c\cdot u_3)\qquad\qquad c\in \{-1,0,1\}.
\end{align*}
The set $\A$ has $12$ two-dimensional faces, each consisting of exactly three points. Up to equivalence, each of these faces has a unique maximal degree two Cayley structure of length $5$. See also Example \ref{ex:cayley} and Figure \ref{fig:cayley}.

All of these Cayley structures give rise to smooth conics in $X_\A$, and hence to irreducible components in the Hilbert scheme of conics in $X_\A$; see Example \ref{ex:fanohilb}. The $9$ length $1$ Cayley structures yield $9$ components of dimension $2$; these components are themselves toric. After appropriate choice of coordinates, the fans corresponding to the toric varieties appearing as the normalizations of these components are pictured in Figure \ref{fig:fans}. See Examples \ref{ex:fanochow} and \ref{ex:fanohilb2}.
The $12$ length $5$ Cayley structures yield $12$ components of dimension $5$. Each of these components is isomorphic to the $\PP^5$ parametrizing conics in the plane.
\end{introex}

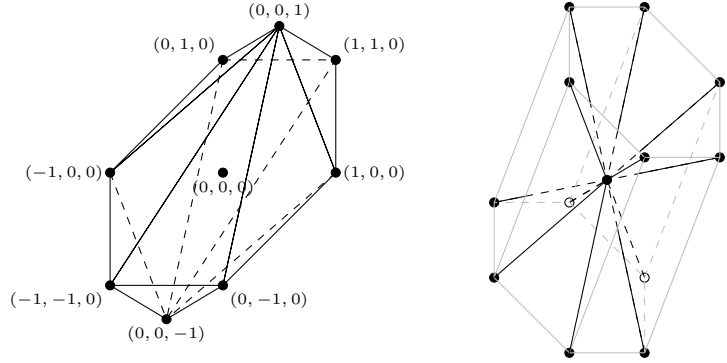
\begin{figure}\tiny{
	\begin{tikzpicture}[scale=1.5]
		\draw[fill] (0,0) circle [radius=0.04]  node[below] {$(0,0,0)$};
		\draw[fill] (-1,0) circle [radius=0.04] node[left] {$(-1,0,0)$};
\draw[fill] (0,-1) circle [radius=0.04] node[below right] {$(0,-1,0)$};
\draw[fill] (-1,-1) circle [radius=0.04] node[below left] {$(-1,-1,0)$};
\draw[fill] (1,0) circle [radius=0.04] node[right] {$(1,0,0)$};
\draw[fill] (0,1) circle [radius=0.04] node[above left] {$(0,1,0)$};
\draw[fill] (1,1) circle [radius=0.04] node[above right] {$(1,1,0)$};
\draw[fill] (.5,1.3) circle [radius=0.04] node[above] {$(0,0,1)$};
\draw[fill] (-.5,-1.3) circle [radius=0.04] node[below] {$(0,0,-1)$};
\draw (0,1) -- (-1,0) -- (-1,-1) -- (0,-1) -- (1,0) -- (1,1) -- (.5,1.3) -- (1,0) -- (.5,1.3) -- (0,-1) -- (.5,1.3) -- (-1,-1) -- (.5,1.3) -- (-1,0) -- (.5,1.3) -- (0,1);
\draw (-1,-1) -- (-.5,-1.3) -- (0,-1);
\draw[dashed] (0,1) -- (1,1) -- (-.5,-1.3);
\draw[dashed] (1,0) -- (-.5,-1.3);
\draw[dashed] (-1,0) -- (-.5,-1.3);
\draw[dashed] (0,1) -- (-.5,-1.3);
	\draw[draw=none] (0,-1.9);
\end{tikzpicture}}
\raisebox{1.5em}{\begin{tikzpicture}[scale=1]
\draw[draw=none] (-2,0);
\draw (-1,1) -- (.5,1.3);
\draw (0,1) -- (.5,1.3);
\draw (1,-1) -- (.5,1.3);
\draw (2,1.6) -- (.5,1.3);
\draw (2,2.6) -- (.5,1.3);
\draw (1,3.6) -- (.5,1.3);
\draw (0,3.6) -- (.5,1.3);
\draw (0,1) -- (.5,1.3);

\draw[fill,white] (-1,0) -- (0,-1) -- (1,1.6) -- (0,2.6) -- (-1,0);
\draw[dashed] (-1,1) -- (.5,1.3);
\draw[dashed] (0,1) -- (.5,1.3);
\draw[dashed] (1,0) -- (.5,1.3);
\draw[dashed] (1,-1) -- (.5,1.3);
\draw[dashed] (2,1.6) -- (.5,1.3);
\draw[dashed] (2,2.6) -- (.5,1.3);
\draw[dashed] (1,3.6) -- (.5,1.3);
\draw[dashed] (0,3.6) -- (.5,1.3);
\draw[dashed] (0,1) -- (.5,1.3);

\begin{scope}[shift={(1,2.6)}]
\draw[fill] (-1,0) circle [radius=0.06];
\draw[fill] (0,-1) circle [radius=0.06] ;
\draw[fill] (1,-1) circle [radius=0.06];
\draw[fill] (1,0) circle [radius=0.06] ;
\draw[fill] (0,1) circle [radius=0.06] ;
\draw[fill] (-1,1) circle [radius=0.06];
\draw[lightgray] (0,1) -- (-1,1) -- (-1,0) -- (0,-1) -- (1,-1) -- (1,0) -- (0,1);
\end{scope}

\begin{scope}
\draw[fill] (-1,0) circle [radius=0.06];
\draw[fill] (0,-1) circle [radius=0.06] ;
\draw[fill] (1,-1) circle [radius=0.06];
\draw[draw, circle] (1,0) circle [radius=0.06] ;
\draw[draw, circle] (0,1) circle [radius=0.06] ;
\draw[fill] (-1,1) circle [radius=0.06];
\draw[lightgray] (-1,1) -- (-1,0) -- (0,-1) -- (1,-1);
\draw[lightgray,dashed] (1,-1) -- (1,0) -- (0,1) -- (-1,1);
\end{scope}
\begin{scope}[shift={(-1,0)}]
	\draw[lightgray] (0,0) -- (1,2.6);
\end{scope}
\begin{scope}[shift={(1,0)}]
	\draw[lightgray,dashed] (0,0) -- (1,2.6);
\end{scope}
\begin{scope}[shift={(0,1)}]
	\draw[lightgray,dashed] (0,0) -- (1,2.6);
\end{scope}
\begin{scope}[shift={(0,-1)}]
	\draw[lightgray] (0,0) -- (1,2.6);
\end{scope}
\begin{scope}[shift={(1,-1)}]
	\draw[lightgray] (0,0) -- (1,2.6);
\end{scope}
\begin{scope}[shift={(-1,1)}]
	\draw[lightgray] (0,0) -- (1,2.6);
\end{scope}
\draw[fill] (.5,1.3) circle [radius=0.06];
\draw (-1,0) -- (.5,1.3);
\draw (0,-1) -- (.5,1.3);
\draw (0,2.6) -- (.5,1.3);
\draw (1,1.6) -- (.5,1.3);
\end{tikzpicture}}

\caption{The set $\A$ for a Fano threefold and the corresponding normal fan (Example \ref{ex:fano}).}\label{fig:introex}
\end{figure}

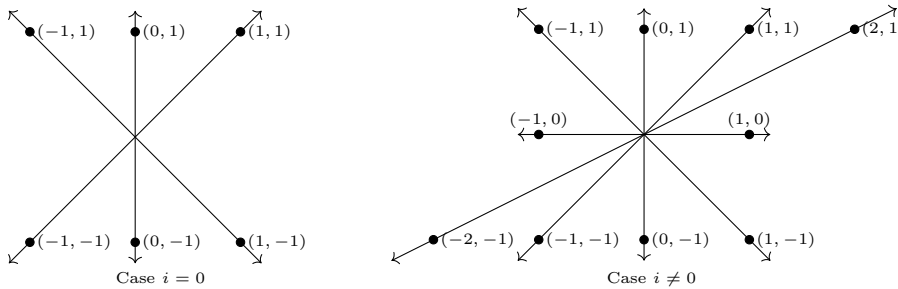
\begin{figure}{\tiny
		\[		\begin{array}{c @{\qquad\qquad}c}
\begin{tikzpicture}[scale=1.4]
\draw[->] (0,0) -- (-1.2,1.2);
\draw[->] (0,0) -- (0,1.2);
\draw[->] (0,0) -- (1.2,1.2);
\draw[->] (0,0) -- (-1.2,-1.2);
\draw[->] (0,0) -- (0,-1.2);
\draw[->] (0,0) -- (1.2,-1.2);
\draw[fill] (0,1) circle [radius=0.04] node[right] {$(0,1)$};
\draw[fill] (-1,1) circle [radius=0.04] node[right] {$(-1,1)$};
\draw[fill] (1,1) circle [radius=0.04] node[right] {$(1,1)$};
\draw[fill] (0,-1) circle [radius=0.04] node[right] {$(0,-1)$};
\draw[fill] (-1,-1) circle [radius=0.04] node[right] {$(-1,-1)$};
\draw[fill] (1,-1) circle [radius=0.04] node[right] {$(1,-1)$};
	\end{tikzpicture}&
	\begin{tikzpicture}[scale=1.4]
		\draw[->] (0,0) -- (-2.4,-1.2);
		\draw[->] (0,0) -- (2.4,1.2);
		\draw[->] (0,0) -- (-1.2,0);
		\draw[->] (0,0) -- (1.2,0);
		\draw[->] (0,0) -- (-1.2,1.2);
\draw[->] (0,0) -- (0,1.2);
\draw[->] (0,0) -- (1.2,1.2);
\draw[->] (0,0) -- (-1.2,-1.2);
\draw[->] (0,0) -- (0,-1.2);
\draw[->] (0,0) -- (1.2,-1.2);
\draw[fill] (1,0) circle [radius=0.04] node[above] {$(1,0)$};
\draw[fill] (-1,0) circle [radius=0.04] node[above] {$(-1,0)$};
\draw[fill] (0,1) circle [radius=0.04] node[right] {$(0,1)$};
\draw[fill] (-1,1) circle [radius=0.04] node[right] {$(-1,1)$};
\draw[fill] (1,1) circle [radius=0.04] node[right] {$(1,1)$};
\draw[fill] (0,-1) circle [radius=0.04] node[right] {$(0,-1)$};
\draw[fill] (-1,-1) circle [radius=0.04] node[right] {$(-1,-1)$};
\draw[fill] (1,-1) circle [radius=0.04] node[right] {$(1,-1)$};
\draw[fill] (2,1) circle [radius=0.04] node[right] {$(2,1)$};
\draw[fill] (-2,-1) circle [radius=0.04] node[right] {$(-2,-1)$};
	\end{tikzpicture}\\
	\textrm{Case $i=0$}& \textrm{Case $i\neq 0$}
\end{array}\]
\caption{Fans for components of the Hilbert scheme of conics (Example \ref{ex:fano}).}\label{fig:fans}
}\end{figure}

    \subsection*{Acknowledgements}
Both authors were supported by NSERC Discovery Grants. We thank Dhruv Ranganathan, Sandra Di Rocco, and Luca Schaffler for helpful discussions. We thank the anonymous referee for insightful comments.

\section{Preliminaries}\label{sec:prelim}
    \subsection{Toric Varieties}\label{sec:toric}
We will always be working over an algebraically closed field $\KK$ of characteristic zero. Fix a lattice $M$. To a finite subset $\A\subset M$, we associate the (not necessarily normal) projective toric variety
\[X_\A=\Proj \KK[S_\A]\subset \PP^{\#\A-1}\]
where $S_\A$ is the semigroup generated by elements $(u,1)\in M\times \ZZ$ for $u\in \A$, and $\KK[S_\A]$ is the corresponding semigroup algebra.
For any subset $\tau$ of $M$, we will use $\langle \tau \rangle$ to denote the sublattice consisting of all differences of elements of $\tau$.

Given $v\in \A$, we denote the associated homogeneous coordinate of $X_\A$ by $x_v$.
The variety $X_\A$ comes equipped with an action of the torus $T_\A=\spec \KK[\langle \A \rangle]$; the action on the projective coordinate $x_v$ has weight $v$.
We will use the notation $T=\spec \KK[M]$ and note that the tori $T_\tau$ (for $\tau$ a subset of $\A$) are quotients of $T$. In particular, $T$ also acts on $X_\A$.
For more details on toric varieties see \cite{cls}.

A \emph{face} $\tau$ of $\A$ is the intersection of $\A$ with a face of the convex hull $\conv \A$, and we write $\tau \preceq \A$.  Note that we consider $\A$ to be a face of itself. 
There is a natural closed embedding $X_\tau\subset X_\A$ determined by the homomorphism $\KK[S_\A]\to \KK[S_\tau]$ which for any $v\in \A$ sends
\begin{equation*}
x_v\mapsto \begin{cases}
	x_v& v\in \tau\\
	0 &v\notin \tau
\end{cases}.
\end{equation*}

Given a face $\tau$ of $\A$, we let $x_\tau$ be the point of $X_\A$ such that 
$x_u=0$ for $u\notin \tau$ and $x_u=1$ for $u\in \tau$. We note that the $T_\tau$ orbit of $x_\tau$ is dense in $X_\tau$. We write $\partial X_\tau$ for the complement of this orbit, i.e. $\partial X_\tau = \bigcup_{\tau' \prec \tau} X_{\tau'}$.

\subsection{Degree-$d$ Cayley Structures}\label{sec:cayley}
We now generalize the notion of a Cayley structure from \cite[Definition 3.1]{fano}. 
For natural numbers $d$ and $\ell$, we set
\[\Delta_\ell(d)= \{(u_0,\ldots,u_\ell)\in \ZZ^{\ell+1}_{\geq 0}\ :\ u_0+\ldots+u_\ell=d\}.\]
Note that the convex hull of $\Delta_\ell(d)$ is the $d$-th dilation of an $\ell$-dimensional standard simplex.
We will use $e_0,\ldots,e_\ell$ to refer to the elements of the standard basis of $\ZZ^{\ell+1}$, and $e_0^*,\ldots,e_\ell^*$ to refer to the dual basis elements.

Let $\tau$ be a face of $\A$.
\begin{defn}\label{defn: Cayley structure}
A \emph{weak Cayley structure} of length $\ell$ and degree $d$ on $\tau$ is a non-constant affine-linear map $\pi:\tau\to \Delta_\ell(d)$.  

Consider any map $\pi:\tau\to\Delta_\ell(d)$.
\begin{enumerate}
	\item We say $i\in \{0,\ldots, \ell\}$ is a basepoint for $\pi$ if for every $w\in\pi(\tau)$, $w_i>0$. We say $\pi$ is \emph{basepoint-free} if 
		no element of $\{0,\ldots,\ell\}$ is a basepoint.
	\item We say that $\pi$ is \emph{concise} if for every $0\leq i \leq \ell$ there exists $v\in \pi(\tau)$ such that $v_i \neq 0$.
\end{enumerate}
A \emph{Cayley structure} is a weak Cayley structure that is basepoint-free and concise.
\end{defn}

Note that in the special case $d=1$, basepoint-freeness and conciseness are equivalent to surjectivity of $\pi$.
We call any two (weak) Cayley structures equivalent if they differ only by a permutation of the basis vectors of $\ZZ^{\ell+1}$. This is the same as identifying any two (weak) Cayley structures differing by an affine automorphism of $\Delta_\ell(d)$; this defines an equivalence relation on the set of (weak) Cayley structures.

Let $\B\subseteq \Delta_\ell(d)$ be the image of a weak Cayley structure $\pi:\tau\to \Delta_\ell(d)$. The map $\pi$ determines a surjective ring homomorphism
\begin{align*}
	\KK[S_\tau]&\to \KK[S_\B]\\
	x_u&\mapsto x_{\pi(u)}
\end{align*}
and hence an embedding $X_\B\hookrightarrow X_\tau$. This induces an inclusion of tori $T_\B \to T_\tau$; we will denote by $T_\pi$ the image of $T_\B$ in $T_\tau$. 
Let $N_\tau=\Hom(\langle \tau \rangle,\ZZ)$ be the cocharacter lattice of $T_\tau$.
The weak Cayley structure $\pi:\tau\to \Delta_\ell(d)$ induces a linear map
\begin{align*}
	\pi^*:(\ZZ^{\ell+1})^*&\to N_\tau
\end{align*}
where for $v\in \ZZ^{\ell+1}$ and $u,w\in \tau$, $\pi^*(v)$ sends $u-w\in\langle \tau\rangle $ to $v(\pi(u)-\pi(w))$.

\begin{defn}
	Let $\pi:\tau\to \Delta_\ell(d)$ be a weak Cayley structure. Let $c \in \mathbb{Z}^{\ell+1}$ be the coordinatewise minimum of $\pi$, that is, $c_i = \min\{\pi(u)_i : u \in \tau\}$. The \emph{resolution} of $\pi$ is the map
	\begin{align*}
\Res(\pi) :\tau&\to\Delta_\ell(d')\\
u&\mapsto \pi(u)- c
	\end{align*}
where $d' = d - \sum c_i$.

\end{defn}
For any weak Cayley structure $\pi$, the resolution $\Res(\pi)$ will be basepoint-free. If $\pi$ is already basepoint-free, $\Res(\pi)=\pi$.

\begin{defn}
Let $\pi:\tau\to \Delta_\ell(d)$ be a weak Cayley structure, and let $F$ be the minimal face of $\Delta_\ell(d)$ containing $\pi(\tau)$.
A \emph{concision} of $\pi$ is the composition of $\pi$ with any affine linear bijection $F\to \Delta_{\ell'}(d)$, where $\ell'=\dim F$.
\end{defn}
For any weak Cayley structure $\pi$, its concisions will all be equivalent. Furthermore, they will be concise. If $\pi$ was basepoint-free, then so is any concision. Hence, given a weak Cayley structure $\pi$, we may always obtain a Cayley structure by considering a concision of $\Res(\pi)$. Note that equivalent Cayley structures have equivalent resolutions and concisions.

\begin{rem}
	The notion of a weak Cayley structure is closely related to the generalized $\pi$-twisted Cayley sums of \cite[Definition 2.2]{pointedfamilies}. Indeed, consider $R$ a $\pi$-twisted Cayley sum over the polytope $F$. If $F$ is a dilate of a standard simplex, then by restricting to the lattice points of $R$, the map $\pi$ gives a Cayley structure. For more general $F$ we may choose an (affine-linear) inclusion of $F$ in a dilate of a standard simplex; by restricting to the lattice points of $R$ we obtain a weak Cayley structure (which may be made into a Cayley structure by taking a concision of its resolution).
\end{rem}

\begin{ex}[Cayley structures on a Fano threefold]\label{ex:cayley}
	We consider the set $\A$ from Example \ref{ex:fano} and pictured in Figure \ref{fig:introex}.
Up to equivalence, there are exactly nine length $1$ degree $2$ Cayley structures defined on the entire set $\A$. The images of the elements of $\A$ under the two Cayley structures
\begin{align*}
\pi:(u_1,u_2,u_3)&\mapsto (1+u_2,1-u_2);\\
\pi':(u_1,u_2,u_3)&\mapsto (1+u_2+u_3,1-u_2-u_3)
\end{align*}
are depicted in the top of Figure \ref{fig:cayley}.
A length $5$ degree $2$ Cayley structure $\pi''$ on the face $\tau=\{(1,0,0),(0,-1,0),(0,0,1)\}$ is depicted in the bottom of Figure \ref{fig:cayley}. Up to equivalence, this is the only length $5$ degree $2$ Cayley structure on $\tau$.

Restricting the Cayley structure $\pi$ to the face $\tau$, we obtain a weak Cayley structure $\pi|_\tau$ that is concise, but not basepoint-free: $1$ is a basepoint. The resolution of $\pi|_\tau$ is the degree $1$ Cayley structure
\[
	\Res(\pi|_\tau):(u_1,u_2,u_3)\mapsto (1+u_2,-u_2).
\]
sending $(1,0,0)$ and $(0,0,1)$ to $e_0$, and $(0,-1,0)$ to $e_1$.

Let $\tau'$ be the face $\tau'=\{(0,-1,0),(0,0,1)\}$. Restricting $\pi''$ to $\tau'$ we obtain a weak Cayley structure $\pi''|_{\tau'}$ that is basepoint-free but not concise. A concision for this weak Cayley structure is given by the length $3$ Cayley structure sending $(0,-1,0)$ to $e_2+e_3$ and $(0,0,1)$ to $e_0+e_1$.
\end{ex}

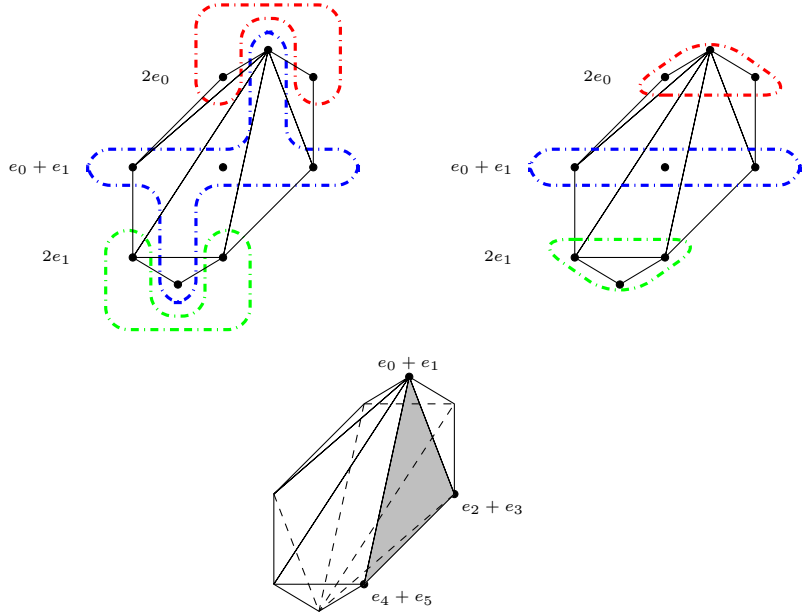
\begin{figure}\tiny{
		\begin{center}\begin{tikzpicture}[scale=1.2]
				\draw[draw=none] (0,-2);
				\draw[dashdotted,blue, very thick,rounded corners=.3cm] (-.7,-1) -- (-.7,-1.5) -- (-.3,-1.5) -- (-.3,-.2) -- (1.5,-.2) -- (1.5,.2) -- (.7,.2) -- (.7,1.5) -- (.3,1.5) -- (.3,.2) -- (-1.5,.2) -- (-1.5,-.2) -- (-.7,-.2) -- (-.7,-1);
\draw[dashdotted,red, very thick,rounded corners=.3cm] (.5,1.8) -- (-.3,1.8) -- (-.3,.7) -- (.2,.7) -- (.2,1.65) -- (.8,1.65) -- (.8,.7) -- (1.3,.7) -- (1.3,1.8) -- (.5,1.8);
\draw[dashdotted,green, very thick,rounded corners=.3cm] (-.5,-1.8) -- (-1.3,-1.8) -- (-1.3,-.7) -- (-.8,-.7) -- (-.8,-1.65) -- (-.2,-1.65) -- (-.2,-.7) -- (.3,-.7) -- (.3,-1.8) -- (-.5,-1.8);
		\draw[fill] (0,0) circle [radius=0.04];  
		\draw[fill] (-1,0) circle [radius=0.04]; 
\draw[fill] (0,-1) circle [radius=0.04]; 
\draw[fill] (-1,-1) circle [radius=0.04];
\draw[fill] (1,0) circle [radius=0.04]; 
\draw[fill] (0,1) circle [radius=0.04];
\draw[fill] (1,1) circle [radius=0.04];
\draw[fill] (.5,1.3) circle [radius=0.04]; 
\draw[fill] (-.5,-1.3) circle [radius=0.04];
\draw (0,1) -- (-1,0) -- (-1,-1) -- (0,-1) -- (1,0) -- (1,1) -- (.5,1.3) -- (1,0) -- (.5,1.3) -- (0,-1) -- (.5,1.3) -- (-1,-1) -- (.5,1.3) -- (-1,0) -- (.5,1.3) -- (0,1);
\draw (-1,-1) -- (-.5,-1.3) -- (0,-1);
\draw (-.5,1) node [left] {$2e_0$};
\draw (-1.6,0) node [left] {$e_0+e_1$};
\draw (-1.6,-1) node [left] {$2e_1$};
\end{tikzpicture}
\qquad\qquad
\begin{tikzpicture}[scale=1.2]
				\draw[draw=none] (0,-2);
		\draw[fill] (0,0) circle [radius=0.04];  
		\draw[fill] (-1,0) circle [radius=0.04]; 
\draw[fill] (0,-1) circle [radius=0.04]; 
\draw[fill] (-1,-1) circle [radius=0.04];
\draw[fill] (1,0) circle [radius=0.04]; 
\draw[fill] (0,1) circle [radius=0.04];
\draw[fill] (1,1) circle [radius=0.04];
\draw[fill] (.5,1.3) circle [radius=0.04]; 
\draw[fill] (-.5,-1.3) circle [radius=0.04];
\draw (0,1) -- (-1,0) -- (-1,-1) -- (0,-1) -- (1,0) -- (1,1) -- (.5,1.3) -- (1,0) -- (.5,1.3) -- (0,-1) -- (.5,1.3) -- (-1,-1) -- (.5,1.3) -- (-1,0) -- (.5,1.3) -- (0,1);
\draw (-1,-1) -- (-.5,-1.3) -- (0,-1);
\draw[dashdotted,red, very thick,rounded corners=.5cm] (0,.8) -- (1.5,.8) -- (.5, 1.5) -- (-.5,.8) -- (0,.8);
\draw[dashdotted,blue, very thick,rounded corners=.3cm] (0,-.2) -- (1.5,-.2) -- (1.5,.2) -- (-1.5,.2) -- (-1.5,-.2) -- (0,-.2);
\draw[dashdotted,green, very thick,rounded corners=.5cm] (-1,-.8) -- (.5,-.8) -- (-.5, -1.5) -- (-1.5,-.8) -- (-1,-.8);
\draw (-.5,1) node [left] {$2e_0$};
\draw (-1.6,0) node [left] {$e_0+e_1$};
\draw (-1.6,-1) node [left] {$2e_1$};
\end{tikzpicture}
	\begin{tikzpicture}[scale=1.2]
\draw[fill] (0,-1) circle [radius=0.04] node[below right] {$e_4+e_5$};
\draw[fill] (1,0) circle [radius=0.04] node[below right] {$e_2+e_3$};
\draw[fill] (.5,1.3) circle [radius=0.04] node[above] {$e_0+e_1$};
\draw[fill,lightgray] (1,0,0) -- (0,-1,0) -- (.5,1.3) ;
\draw (0,1) -- (-1,0) -- (-1,-1) -- (0,-1) -- (1,0) -- (1,1) -- (.5,1.3) -- (1,0) -- (.5,1.3) -- (0,-1) -- (.5,1.3) -- (-1,-1) -- (.5,1.3) -- (-1,0) -- (.5,1.3) -- (0,1);
\draw (-1,-1) -- (-.5,-1.3) -- (0,-1);
\draw[dashed] (0,1) -- (1,1) -- (-.5,-1.3);
\draw[dashed] (1,0) -- (-.5,-1.3);
\draw[dashed] (-1,0) -- (-.5,-1.3);
\draw[dashed] (0,1) -- (-.5,-1.3);
\end{tikzpicture}
\end{center}}
\caption{Three Cayley structures for a Fano threefold (Example \ref{ex:cayley}). On top, elements with common image under the Cayley structure are grouped by the dashed lines. On the bottom, the Cayley structure is defined only on the highlighted face. Elements are labeled by their images under the Cayley structure.}\label{fig:cayley}
\end{figure}

\section{Cayley Structures and Rational Curves}\label{sec:curves}
\subsection{Constructing Curves}\label{sec:curveconst}
Fix a natural number $\ell$. For $i=0,\ldots, \ell$, we consider linear forms $f_i\in \KK[y_0,y_1]$. We will write
\[
\ff=(f_0,\ldots,f_\ell)
\]
and for $v=(v_0,\ldots,v_\ell)\in\ZZ^{\ell+1}$ set
\[
	\ff^v=\prod_{i=0}^\ell f_i^{v_i}.
\]
\begin{defn}[Rational curve from a weak Cayley structure] \label{defn:cayley-curve}
Let $\pi:\tau\to \Delta_\ell(d)$ be a weak Cayley structure. Let $\ff$ be a tuple of linear forms. We define
\begin{align*}
	\rho_{\pi,\ff}:\PP^1 &\dashrightarrow X_\tau\subset \PP^{\#\A-1}, \\
	y &\mapsto \left( \ff^{\pi(u)}\right)_{u\in \A},
\end{align*}
where we adopt the convention that for $u\notin\tau$, $\ff^{\pi(u)}:= 0$. We denote the closure of the image of $\rho_{\pi,\ff}$ by $C_{\pi,\ff}$. Note that $C_{\pi, \ff} \subseteq X_\tau$ because $\pi$ is affine-linear.
\end{defn}

It is straightforward to observe that $C_{\pi,\ff}=C_{\Res(\pi),\ff}$.
In the special case that $\pi$ has length $1$, the curve $C_{\pi,\ff}$ is independent of $\ff$ as long as its entries have distinct roots. In this case, we may simply write $C_\pi$ for the image of $\rho_{\pi,\ff}$.

\begin{prop}\label{prop:curve}
Let $\pi$ and $\ff$ be as in Definition \ref{defn:cayley-curve}. Suppose the entries of $\ff$ have distinct roots and $\pi$ is basepoint-free. 
\begin{enumerate}
	\item The map $\rho_{\pi, \ff}$ is a basepoint-free morphism.\label{itemi}
	\item If in addition $\pi$ is concise, the preimage of the boundary, $\rho_{\pi, \ff}^{-1}(\partial X_\tau)$, is the set of roots of the $f_i$.\label{itemii}
\end{enumerate}
\end{prop}
\begin{proof}
	Let $y \in \PP^1$. Since we assumed the entries of $\ff$ have distinct roots, at most one $f_i$ vanishes at $y$, so let $i$ be such that $f_j(y) \ne 0$ for all $j \ne i$. Since $\pi$ is basepoint-free, there exists $u \in \A$ such that $\pi(u)_i = 0$, that is, $f_i \nmid f^{\pi(u)}$. Then $\ff^{\pi(u)}(y) \ne 0$ as it is a product of nonvanishing forms. This shows (\ref{itemi}). For (\ref{itemii}), we have $\rho_{\pi, \ff}(y) \in \partial X_\tau$ if and only if $\ff^{\pi(u)}(y) = 0$ for some $u$. Since $\pi$ is concise, this holds if and only if $y$ is a root of some $f_i$.
\end{proof}

For the remainder of this section, we will assume that $\pi$ is a Cayley structure, that is, is both basepoint-free and concise. We may reduce to this situation by replacing a weak Cayley structure by a concision of its resolution.
By (\ref{itemii}), the image curve $C_{\pi, \ff}$ intersects the dense torus orbit of $X_\tau$.
By acting on the curve by the torus $T_\tau$, we obtain a family of curves $\{t\cdot C_{\pi,\ff}\}_{t\in T_\tau}$. Note that the roots of the $f_i$ (as in Proposition \ref{prop:curve}(\ref{itemii})) give well-defined marked points on $\PP^1$, but determine the forms $f_i$ only up to scalar multiple. These scalars are captured by the action of the torus $T_\tau$:

\begin{lemma}\label{lem:cayley-scalars}
Let $\ff = (f_0, \ldots, f_\ell)$ and let $\tilde \ff = (c_0 f_0, \ldots, c_\ell f_\ell)$ where $c_i \in \KK^*$ for all $i$. Then $\rho_{\pi, \ff} = t \cdot \rho_{\pi, \tilde \ff}$ for some unique $t \in T_\tau$.
\end{lemma}
\begin{proof}
We have $\rho_{\pi, \ff}(y) = (\ff^{\pi(u)} : u \in \tau)$ and $\rho_{\pi, \tilde\ff}(y) = (\tilde\ff^{\pi(u)} : u \in \tau)$. For each $u \in \tau$, let $\lambda_u = (c_0, \ldots, c_\ell)^{\pi(u)} \in \KK^*$, so
\begin{equation}\label{eqn:cayley-scalars}
\tilde \ff^{\pi(u)} = \lambda_u \ff^{\pi(u)}.
\end{equation}
Note that both $\rho_{\pi, \ff}$ and $\rho_{\pi, \tilde \ff}$ satisfy the defining equations of $X_\tau$. In particular, consider an affine relation $\sum_{u \in \tau} a_u u = \sum_{u \in \tau} b_u u$ for some coefficients $a_u, b_u \in \NN$. Applying this to Equation \ref{eqn:cayley-scalars}, we see that the point $\lambda = (\lambda_u : u \in \tau) \in \PP^{|\tau|-1}$ itself satisfies these relations, i.e. is a point of the interior of $X_\tau$ and so a translate of the distinguished point $x_\tau \in X_\tau$. That is, $\lambda = t \cdot x_\tau$ for some unique $t \in T_\tau$.
\end{proof}

Recall that $M_{0,\ell+1}$ is the moduli space of $\ell+1$ distinct marked points on $\PP^1$ up to automorphism. 

\begin{defn}[Family of curves from a Cayley structure]\label{defn:little}
Let $\bp = (P_0, \ldots, P_\ell) \in M_{0, \ell+1}$. We choose coordinates $P_0 = (1 : 0)$ and $P_i = (c_i : 1)$ for $i > 0$, with $c_1 = 0, c_2 = 1$. We let $\ff(\bp) = (y_1, y_0, y_0 - y_1, \ldots, y_0 - c_\ell y_1)$ and define
\begin{align*}
\rho_\pi : M_{0,\ell+1} \times T_\tau \times \PP^1 \to X_\tau,\\
\rho_\pi(\bp, t, s) = t \cdot \rho_{\pi, \ff(\bp)}(s).
\end{align*}
When $\ell=1$, we abuse notation and set $M_{0,\ell+1}=\spec \KK$, and $\ff(\bp)=(y_1,y_0)$.
\end{defn}
We show below in Proposition \ref{prop:curvetocayley} that the family $\rho_\pi$ induces a quasifinite map from $M_{0,\ell+1} \times T_\tau$ to the Hilbert scheme as long as $\ell>1$, which we use for dimension counting. 

If $\B\subseteq \Delta_\ell(d)$ is the image of a Cayley structure $\pi$, we may also consider the Cayley structure $\iota:\B\to \Delta_\ell(d)$, where $\iota$ is just the inclusion. The curve $C_{\B,\ff}:=C_{\iota,\ff}$ is an isomorphic linear projection of $C_{\pi,\ff}$. Hence, when considering properties of $C_{\pi,\ff}$ such as degree or arithmetic genus, we may consider instead the corresponding properties of $C_{\B,\ff}$.
We will also denote the map $\rho_{\iota,\ff}$ by $\rho_{\B,\ff}$.

\subsection{From Curve to Cayley Structure}\label{sec:curvetocayley}
We show that all rational curves in $X_\A$ arise via the above construction.

\begin{prop}\label{prop:curvetocayley}
Consider a rational curve $C\subseteq X_\tau$ of degree $d$ that intersects the dense torus orbit of $X_\tau\subseteq X_\A$. 
\begin{enumerate}
	\item
		There exists  $\ell\in \NN$,  a Cayley structure $\pi:\tau\to \Delta_\ell(d)$, a point $\bp\in M_{0,\ell+1}$, and $t \in T_\tau$ such that 
	\[
		C=t\cdot C_{\pi, \ff(\bp)}.
	\]
\item The Cayley structure $\pi$ and $\bp\in M_{0,\ell+1}$ are uniquely determined up to permutation by an element of the symmetric group $S_{\ell+1}$. 
\item	If $\ell>1$, there are only finitely many $t\in T_\tau$ such that $t \cdot C = C$.
\end{enumerate}
\end{prop}
\noindent One may prove this proposition using Cox's notion of a $\Delta$-collection to describe maps from $\PP^1$ to $X_\tau$ in terms of the Cox ring of a resolution of $X_\tau$ \cite{functor}. This is the approach used in \cite{stablemaps} to describe genus zero log stable maps to a toric variety. Here, we will instead take a direct, more elementary approach.
\begin{proof}
	Consider the normalization $\PP^1 \to C$ of $C$. We consider the preimage $\{ P_0, \ldots, P_\ell\}$ of the finite, nonempty set $C \cap \partial X_\tau$ under this map. This defines $\ell$; after ordering the elements of this preimage, we obtain $\bp\in M_{0,\ell+1}$. We choose coordinates on $\PP^1$ so that $P_0 = (1 : 0) \in \PP^1$ and $P_i = (c_i : 1)$ for $i = 1, \ldots, \ell$. We put $\ff(\bp) = (y_1, y_0, y_0-y_1, \ldots, y_0 - c_\ell y_1)$ as in Definition \ref{defn:little}.

The normalization map $\PP^1 \to X_\tau$ has degree one and is thus given by a tuple of forms $(F_u \in \KK[y_0, y_1]_d : u \in \tau)$. We have $F_u \not\equiv 0$ for each $u$ since $C$ intersects the dense torus orbit of $X_\tau$. The factors of $F_u$ are all from $\ff(\bp)$ up to scalar, so we define $\pi : \tau \to \Delta_\ell(d)$ by
\[F_u = \lambda_u \ff^{\pi(u)}\]
for each $u$ and for some tuple of constants $\lambda = (\lambda_u \in \KK^* : u \in \tau)$. Our choice of $P_i$'s implies that $\pi$ is basepoint-free and concise (cf.~Definition \ref{defn: Cayley structure}). To see that $\pi$ is affine-linear, consider an affine relation $\sum a_u u = \sum b_u u$ for some coefficients $a_u, b_u \in \NN$. Then $X_\tau$ has the defining equation $\prod x_u^{a_u} = \prod x_u^{b_u}$. Since $C \subseteq X_\tau$, we may pull this back to $\PP^1$ to obtain
\[\prod (\lambda_u \ff^{\pi(u)})^{a_u} = \prod (\lambda_u \ff^{\pi(u)})^{b_u}.\]
Since the factors $\ff$ are all distinct, by uniqueness of factorization we have $\sum a_u \pi(u) = \sum b_u \pi(u)$. This shows $\pi$ is affine-linear, so $\pi$ is a Cayley structure. The argument from the proof of Lemma \ref{lem:cayley-scalars} then shows that there exists a unique $t \in T_\tau$ such that $\PP^1 \to C$ is exactly the map $t \cdot \rho_{\pi, \ff(\bp)}$.

This shows the existence of $\ell$, $\pi$, $\bp$, and $t$. 
For any $\ell'$, $\pi'$ of degree $d$, $\bp'$, and $t'$ satisfying $C=t'\cdot C_{\pi',\ff(\bp)}$, the map $t\cdot \rho_{\pi,\ff(\bp)}:\PP^1 \to C$ must also be the normalization of $C$. The uniqueness claims follow.

Consider now all $t\in T_\tau$ such that $t\cdot C=C$. If this is infinite, it contains a one-parameter subgroup, and $C$ is a torus translate of the orbit closure of this subgroup. In particular, $C$ is a torus translate of a \emph{toric} curve, and is thus parametrized by monomials. It follows that $\ell=1$.
\end{proof}

\subsection{Stabilizers}

For a rational curve $C \subseteq X_\tau$, any $t \in T_\tau$ such that $t \cdot C = C$ induces a permutation of the $\ell+1$ points of the normalization of $C$ lying over $C \cap \partial X_\tau$. If $\ell > 1$, this permutation determines $t$ as an automorphism of the curve, hence identifies the stabilizer of $C$ with a subgroup of $S_{\ell+1}$ via its action on $M_{0, \ell+1}$. 

\begin{rem}\label{rem:le1}
In the setting of Proposition \ref{prop:curvetocayley}, when $\ell=1$ the stabilizer in $T_\tau$ of the curve $C$ consists exactly of the one-dimensional torus $T_\pi$ (and $C$ is a translate of the closure of $T_\pi$).
\end{rem}

We briefly examine the stabilizers when $\ell > 1$. Note that the finite subgroups of $PGL_2(\KK)$ are cyclic, dihedral, or $A_4, S_4$ or $A_5$. The stabilizer of $C$ is moreover abelian since it is also a subgroup of $T_\tau$, so it is either cyclic or $\mu_2\times \mu_2$. In any case, the existence of a nontrivial stabilizer highly constrains the Cayley structure, the choice of marked points and the cycle type of the resulting permutation $\sigma$.

\begin{prop} \label{prop:stabilizer-ell-ge-1}
Let $\pi : \tau \to \Delta_\ell(d)$ be a Cayley structure, with $\ell > 1$. Let $\sigma \in S_{\ell+1}$ be a permutation. Then $\rho_\pi$ admits fibers $C_{\pi, \ff(\bp)}$ with stabilizers containing $\sigma$ if and only if the following holds:
\begin{enumerate}
	\item $\pi \circ \sigma = \pi$, and \label{stabi}
	\item the cycle type of $\sigma$ is $(1, 1, k, \ldots, k)$ for some $k$. \label{stabii}
\end{enumerate}
Without loss of generality assume $\sigma$ fixes $0$ and $1$ and let $\zeta_k$ denote a $k$-th root of unity. The corresponding fibers are given by putting $P_0 = (1:0)$, $P_1 = (0:1)$, and the remaining marked points in disjoint orbits each of the form \[\{(1:\zeta_k^i c)\ |\ 0 \leq i < k\}\] for some $c \in \KK^*$.
\end{prop}

\begin{cor}\label{cor:stab}
If $\ell > 1$, $C_{\pi, \ff}$ has trivial stabilizer for general $\ff$.
\end{cor}

\begin{proof}[Proof of Proposition \ref{prop:stabilizer-ell-ge-1}]
Let $\rho_{\pi, \ff(\bp)} : \PP^1 \to C_{\pi,\ff(\bp)}$ be a rational curve using $\pi$. Let $t \in T_\tau$ be such that $t \cdot C_{\pi, \ff(\bp)} = C_{\pi, \ff(\bp)}$. There is a unique $\phi : \PP^1 \to \PP^1$ such that $t \cdot \rho_{\pi, \ff(\bp)} = \rho_{\pi, \ff(\bp)} \circ \phi$ as maps. As discussed above, $\phi$ induces a permutation $\sigma$ of $\{0, \ldots, \ell\}$ and likewise an automorphism of $\Delta_\ell(d)$.

Claim (\ref{stabi}) follows by examining factors of the forms $F_u$ as in the proof of  Lemma \ref{prop:curvetocayley}, using uniqueness of factorization.

For claim (\ref{stabii}), we use the fact that an order-$k$ automorphism of $\PP^1$ is conjugate to multiplication by $\zeta_k$. In particular, $\phi$ has two fixed points, which must lie over $C \cap \partial X_\tau$ since $T_\tau$ acts freely on the interior of $X_\tau$. That is, the fixed points of $\phi$ are among the marked points. All other orbits of $\phi$, hence of $\sigma$, are of size $k$ and of the stated form. This shows (\ref{stabii}) and the last claim.

Conversely, suppose (\ref{stabi})-(\ref{stabii}) hold and the marked points are chosen as described; we check that $\sigma$ arises from a torus element. We denote the marked points as $P_0 = (1 : 0), P_1 = (0: 1)$ and $P_{ij} = (\zeta_k^i c_j : 1)$ for some choices of $c_j$ and for $0 \leq i < k$. Note that
\[\prod_{0 \leq i < k} (y_0 - \zeta_k^i c_j y_1) = y_0^k - c_j^k y_1^k.\]
By abuse of notation, we write $\pi(u)_{*j}$ for the common value of $\pi(u)_{ij}$ for all $i$. Then
\begin{align*}
F_u(y_0, y_1) = y_1^{\pi(u)_0} y_0^{\pi(u)_1} \prod_{i,j} (y_0 - \zeta_k^i c_j y_1)^{\pi(u)_{ij}}
= y_1^{\pi(u)_0} y_0^{\pi(u)_1} \prod_j (y_0^k - c_j^k y_1^k)^{\pi(u)_{*j}}.
\end{align*}
Applying the substitution $y_0 \mapsto \zeta_k y_0$, we find
\[
F_u(\zeta_k y_0, y_1) = \zeta_k^{\pi(u)_1} F_u(y_0, y_1).
\]
In particular, setting $t_u := \zeta_k^{\pi(u)_1}$, the tuple $(t_u\ |\ u \in \tau)$ satisfies the defining relations of $X_\tau$, since $\pi$ does, hence corresponds to an element $t \in T_\tau$. We have $t \cdot \rho_{\pi, \ff} = \rho_{\pi, \ff} \circ \phi$, as required.
\end{proof}

\begin{rem}
	A stabilizer isomorphic to $\mu_2\times \mu_2$ can be represented by the automorphisms $z \mapsto z, -z, 1/z, -1/z$ of $\PP^1$. For such $C_{\pi, \ff}$, the non-free orbits $\{0, \infty\}$, $\{1, -1\}$ and $\{i, -i\}$ must all be among the marked points, since $T_\tau$ acts freely on the interior of $X_\tau$. Any additional marked points must come in disjoint $4$-tuples of the form $\{c, -c, 1/c, -1/c\}$ for various $c$. The Cayley structure itself must also satisfy strong constraints: for indices $j,k$ corresponding to any of the pairs $\{0, \infty\}$, $\{1, -1\}$ and $\{i, -i\}$ we must have $\pi(u)_j = \pi(u)_k$ for all $u$. The remaining marked points come in $4$-tuples $j,k,l,m$ such that  $\pi(u)_j = \pi(u)_k = \pi(u)_l=\pi(u)_m$ for all $u$.

\end{rem}

Later we will need to distinguish between $T_\tau$ and the torus acting on $C_{\pi, \ff}$ when the stabilizer is nontrivial.

\begin{defn}\label{defn:stablattice}
	Let $\pi : \tau \to \Delta_\ell(d)$ be a Cayley structure and $\ff$ a choice of forms. Let $\Stab_{\pi, \ff} \subset T_\tau$ be the corresponding stabilizer and $\overline T_{\pi, \ff} := T_\tau / \Stab_{\pi, \ff}$ the quotient torus. We denote the cocharacter lattice of $\overline T_{\pi, \ff}$ by $N_{\pi,\ff}$. 
\end{defn}
We have a natural map
\[
N_\tau \to N_{\pi, \ff}.
\]
When $\ell=1$, this map is surjective with kernel $\Hom({\GG}_m, T_\pi) \cong \ZZ$, and $N_{\pi,\ff}$ may be identified with the quotient of $N_\tau$ by the image of $\pi^*$.

When $\ell > 1$, this map is an inclusion of lattices with index $k=|\Stab_{\pi, \ff}|$. 
The map is dual to an inclusion of lattices 
\[
	M_{\pi,\ff} \to M_{\tau}=\langle \tau \rangle.
\]
When $\Stab_{\pi, \ff}$ is cyclic, $M_{\pi,\ff}$ is the kernel
of the map
\begin{align*}
M_\tau&\to \ZZ/k\ZZ\\
u-v&\mapsto \pi(u)_1-\pi(v)_1
\end{align*}
where we have ordered coordinates such as in the statement of Proposition \ref{prop:stabilizer-ell-ge-1}. In the case of $\mu_2\times \mu_2$ stabilizer, we obtain $M_\tau$ by intersecting the kernels of two such maps.

\section{Nodes and Cusps}\label{sec:nodesandcusps}
\subsection{Primitive and Smooth Cayley Structures}
In this section, we seek to clarify when a general rational curve determined by a degree $d$ Cayley structure has degree $d$ and is smooth. 
\begin{defn}[Primitive Cayley structures]\label{defn:primitive}
A Cayley structure $\pi:\tau\to\Delta_\ell(d)$ is \emph{imprimitive} if either 
	\begin{enumerate}
		\item $\dim \pi(\tau)=1$ and $\ell>1$; or
		\item $\ell=1$ and $\langle \pi(\tau) \rangle \subset m\cdot \ZZ^2$ for some $m>1$.
	\end{enumerate}
	Note that $\dim \pi(\tau) = 1$ in either case. We say that $\pi$ is \emph{primitive} if it is not imprimitive.
\end{defn}
Consider an imprimitive Cayley structure $\pi$ of degree $d$ with image $\B$. Let $d'$ be the lattice length of the convex hull of $\B$ with respect to the one-dimensional lattice $\langle \B \rangle$. Then up to transposition of $e_0,e_1$, there is a unique affine-linear inclusion $\B\to \Delta_1(d')$. The \emph{reduction} of $\pi$ is the composition
\[
\Red(\pi) : \tau\to \B\to \Delta_1(d').
\]
It is straightforward to see that this is a primitive Cayley structure, and is well-defined up to equivalence.
The \emph{multiplicity} of the imprimitive Cayley structure $\pi$ is the ratio $d/d'$.
If $\pi$ is already primitive, we define it to be its own reduction and to have multiplicity one.

\begin{thm}\label{thm:primitive}
	Let $\pi$ be a degree $d$ Cayley structure. Then $\deg C_{\pi,\ff} = d$ for general choice of $\ff$ if and only if $\pi$ is primitive. If $\pi$ is imprimitive, $C_{\pi,\ff}=C_{\Red(\pi)}$ and $\deg C_{\pi,\ff}$ is the degree of $\Red(\pi)$.
\end{thm}
\noindent We complete the proof of this theorem in \S\ref{sec:nodes}.
\begin{ex}
	Continuing Example \ref{ex:cayley}, recall that a concision of the restriction of $\pi''$ to $\tau'$ was the length $\ell=3$ degree two Cayley structure sending $(0,-1,0)$ to $e_2+e_3$ and $(0,0,1)$ to $e_0+e_1$. Since this Cayley structure has one-dimensional image but $\ell>1$, we see that it is imprimitive. A reduction of this Cayley structure is the map taking $(0,-1,0)$ to $e_0$ and $(0,0,1)$ to $e_1$, which has degree one. The original imprimitive Cayley structure had multiplicity $2$.

	By Theorem \ref{thm:primitive}, we see that the curves coming from the restriction of $\pi''$ to $\tau'$ are actually just lines, with the parametrizing map having degree $2$.
\end{ex}

We now let $\pi:\tau\to\Delta_\ell(d)$ be a primitive Cayley structure.
\begin{defn}[Cuspidal Cayley structures]\label{defn:cuspidal}
	The Cayley structure $\pi$  is \emph{cuspidal} if there exists $i\in\{0,\ldots,\ell\}$ and $v \in \pi(\tau)$ such that $v_i = 0$ and $v'_i > 1$ for all $v' \ne v$.

\end{defn}
For a tuple of forms $\ff$ with roots $P_0,\ldots,P_\ell$, this condition says $\ff^v(P_i) \ne 0$, and that $\ff^{v'}$ vanishes to order at least two at $P_i$ for all $v' \ne v$. 

\begin{defn}\label{defn:nodal}
The Cayley structure $\pi$  is \emph{nodal} if either:
\begin{enumerate}
\item there exists $0\leq i< j \leq \ell$ and $v \in \pi(\tau)$ such that $v_i = v_j = 0$, and $v'_i, v'_j > 0$ for all $v' \ne v$.

\item $\dim \pi(\tau)=2$, $e_i-e_j\notin \langle \pi(\tau)\rangle$ for all $i\neq j$, and up to permutation of the coordinates, $\langle \pi(\tau)\rangle$ is not one of the exceptional lattices listed in Table \ref{table:secant}.
\end{enumerate}
\end{defn}
For a tuple of forms $\ff$, the first condition of the definition says $\ff^v(P_i), \ff^v(P_j) \ne 0$, and that $\ff^{v'}(P_i) = \ff^{v'}(P_j) = 0$ for all $v' \ne v$. The second condition of the definition and the lattices of Table \ref{table:secant} are discussed in \S\ref{sec:nodes2}, see in particular Theorem \ref{thm:ranktwo}.

\begin{table} Lattices generated by:
	\begin{enumerate}
	\item $(1,1,-1,-1,0,0),(1,1,0,0,-1,-1)\in \ZZ^6$
	\item $(2,-1,-1,0,0),(2,0,0,-1,-1)\in\ZZ^5$
	\item $(2,-2,0,0),(2,0,-1,-1)\in \ZZ^4$
	\item $(2,-2,0),(2,0,-2)\in \ZZ^3$
	\item $(2,-2,0),(2,-1,-1)\in \ZZ^3$
\end{enumerate}
	
\caption{Exceptional lattices generated by images of non-nodal Cayley structures. See Definition \ref{defn:nodal}.}\label{table:secant}
\end{table}
\begin{defn}[Smooth Cayley structures]\label{defn:smooth}
	The Cayley structure $\pi$  is \emph{smooth} if it is neither nodal nor cuspidal.
\end{defn}
\noindent See Figure \ref{fig:cuspidal} for examples of the images of imprimitive, cuspidal, and nodal Cayley structures.

\begin{figure}
	\begin{tabular}{c@\qquad c@\qquad c@\qquad c}
		\begin{tikzpicture}[scale=.5]
			\draw[gray] (0,0) -- (0,3);
			\draw[gray] (0,0) -- (3,0);
			\draw[gray] (3,0) -- (0,3);
	\draw[] (0,0) circle [radius=0.1];
	\draw[] (1,0) circle [radius=0.1];
	\draw[] (2,0) circle [radius=0.1];
	\draw[] (3,0) circle [radius=0.1];
	\draw[] (0,1) circle [radius=0.1];
	\draw[] (1,1) circle [radius=0.1];
	\draw[] (2,1) circle [radius=0.1];
	\draw[] (0,2) circle [radius=0.1];
	\draw[] (1,2) circle [radius=0.1];
	\draw[] (0,3) circle [radius=0.1];
	\draw[fill] (0,0) circle [radius=0.1];
	\draw[fill] (1,2) circle [radius=0.1];
\end{tikzpicture}
		&
		\begin{tikzpicture}[scale=.5]
			\draw[gray] (0,0) -- (0,3);
			\draw[gray] (0,0) -- (3,0);
			\draw[gray] (3,0) -- (0,3);
	\draw[] (0,0) circle [radius=0.1];
	\draw[] (1,0) circle [radius=0.1];
	\draw[] (2,0) circle [radius=0.1];
	\draw[] (3,0) circle [radius=0.1];
	\draw[] (0,1) circle [radius=0.1];
	\draw[] (1,1) circle [radius=0.1];
	\draw[] (2,1) circle [radius=0.1];
	\draw[] (0,2) circle [radius=0.1];
	\draw[] (1,2) circle [radius=0.1];
	\draw[] (0,3) circle [radius=0.1];
	\draw[fill] (0,0) circle [radius=0.1];
	\draw[fill] (1,2) circle [radius=0.1];
	\draw[fill] (1,0) circle [radius=0.1];
\end{tikzpicture}
&
\begin{tikzpicture}[scale=.5]
			\draw[gray] (0,0) -- (0,3);
			\draw[gray] (0,0) -- (3,0);
			\draw[gray] (3,0) -- (0,3);
	\draw[] (0,0) circle [radius=0.1];
	\draw[] (1,0) circle [radius=0.1];
	\draw[] (2,0) circle [radius=0.1];
	\draw[] (3,0) circle [radius=0.1];
	\draw[] (0,1) circle [radius=0.1];
	\draw[] (1,1) circle [radius=0.1];
	\draw[] (2,1) circle [radius=0.1];
	\draw[] (0,2) circle [radius=0.1];
	\draw[] (1,2) circle [radius=0.1];
	\draw[] (0,3) circle [radius=0.1];
	\draw[fill] (0,0) circle [radius=0.1];
	\draw[fill] (1,2) circle [radius=0.1];
	\draw[fill] (2,1) circle [radius=0.1];
\end{tikzpicture}
&
\begin{tikzpicture}[scale=.5]
			\draw[gray] (0,0) -- (0,3);
			\draw[gray] (0,0) -- (3,0);
			\draw[gray] (3,0) -- (0,3);
	\draw[] (0,0) circle [radius=0.1];
	\draw[] (1,0) circle [radius=0.1];
	\draw[] (2,0) circle [radius=0.1];
	\draw[] (3,0) circle [radius=0.1];
	\draw[] (0,1) circle [radius=0.1];
	\draw[] (1,1) circle [radius=0.1];
	\draw[] (2,1) circle [radius=0.1];
	\draw[] (0,2) circle [radius=0.1];
	\draw[] (1,2) circle [radius=0.1];
	\draw[] (0,3) circle [radius=0.1];
	\draw[fill] (0,0) circle [radius=0.1];
	\draw[fill] (3,0) circle [radius=0.1];
	\draw[fill] (0,3) circle [radius=0.1];
\end{tikzpicture}
\\
\\
Imprimitive & Cuspidal & Nodal & Nodal
\end{tabular}
\caption{Example images of imprimitive, cuspidal, and nodal Cayley structures with $d=3$, $\ell=2$ (Definitions \ref{defn:primitive}, \ref{defn:cuspidal}, \ref{defn:nodal}).}\label{fig:cuspidal}
\end{figure}
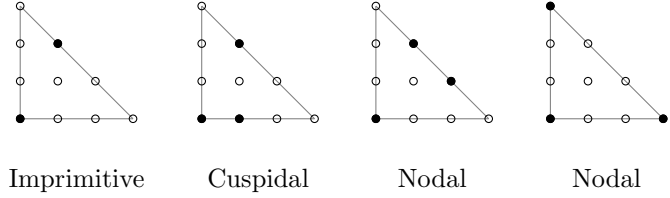

Any singular point of a curve with a single branch we call \emph{cuspidal}. We call singular points with multiple branches \emph{nodal}.
\begin{rem}
The definition of cuspidal Cayley structure says $C_{\pi, \ff}$ has a cusp at the image of $P_i$ (for some $i$). Likewise, the first case of the definition of nodal Cayley structure says $C_{\pi, \ff}$ has a node at the common image of $P_i$ and $P_j$ for some $i \ne j$. We will call these \emph{marked} cusps and nodes. Only in the second type of nodal Cayley structure is the singularity unmarked.
\end{rem}
\begin{thm}\label{thm:smooth}
Let $\pi$ be a primitive Cayley structure. For general choice of $\ff$,
\begin{enumerate}
	\item the curve $C_{\pi,\ff}$ has a cuspidal singularity if and only if $\pi$ is cuspidal;
	\item the curve $C_{\pi,\ff}$ has a nodal singularity if and only if $\pi$ is nodal;
	\item the curve $C_{\pi,\ff}$ is smooth if and only if $\pi$ is smooth.
\end{enumerate}
\end{thm}
\noindent We will complete the proof of this theorem in \S \ref{sec:nodes} and \S\ref{sec:nodes2}.

\begin{rem} \label{rem:image-of-cayley-structure}
We will say a subset $\B \subseteq \Delta_\ell(d)$ is primitive (resp. cuspidal, nodal, smooth) if the inclusion $\B \hookrightarrow \Delta_\ell(d)$ is. For a general Cayley structure $\pi$, each of these properties depends only on the image $\B = \pi(\tau)$, so $\pi$ is primitive (resp. cuspidal, nodal, smooth) if and only if $\B$ is. Indeed, as noted in \S \ref{sec:curveconst}, $C_{\B, \ff}$ is in general an isomorphic linear projection of $C_{\pi, \ff}$.
\end{rem}

\subsection{Setup and Projections}
We describe our approach to proving Theorems \ref{thm:primitive} and \ref{thm:smooth}.
Let $\B\subset \Delta_\ell(d)$ be the image of a Cayley structure $\pi$. 
We may simply assume $\pi$ is the inclusion $\mathcal{B} \hookrightarrow \Delta_\ell(d)$, following Remark \ref{rem:image-of-cayley-structure}.

We denote by $C_d$ the rational normal curve in $\PP^d$. The map $\rho_{\B,\ff}$ arises as the composition of $\PP^1\to C_d\subseteq \PP^d$ with a linear projection $\PP^d\dashrightarrow \PP^{\#\B-1}$. The center $L$ of this projection is the projectivization of the orthogonal complement of 
\[
	\langle \ff^v\ :\ v\in \B\rangle \subseteq \KK[y_0,y_1]_d.
\]
Since the projection is basepoint-free, $L$ does not intersect $C_d$.

Since $\Char \KK=0$, $\rho_{\B,\ff}$ is generically unramified, so $\deg C_{\B,\ff}=d$ if and only if  $\rho_{\B,\ff}$ is birational. This is equivalent to $L$ intersecting only finitely many secant or tangent lines of $C_d$. When this holds, $C_{\B,\ff}$ has a nodal singularity if and only if $L$ intersects a secant line of $C_d$. Likewise, $C_{\B,\ff}$ has a cuspidal singularity if and only if $L$ intersects a tangent line of $C_d$. Thus, in the following we will be analyzing the intersection behaviour of $L$ with secants and tangents of $C_d$ for generic choice of $\ff$.

First we will deal with the case of imprimitive Cayley structures.
\begin{lemma}\label{lemma:primitive}
	If $\pi$ is imprimitive, then for any $\ff$ with distinct roots, $C_{\pi,\ff}=C_{\Red(\pi)}$. Moreover, $\deg C_{\B,\ff}<d$.
\end{lemma}
\begin{proof}
	In both cases in the definition of imprimitive, $\dim \B=1$, so the curve $C_{\B,\ff}$ is just the toric variety $X_\B$.
Let $d'$ be the length of $\conv \B$ with respect to the lattice $\langle \B \rangle$, and let $\B'$ be the image of $\B$ in $\Delta_1(d')$ as in the definition of the reduction. The variety $X_\B$ is projectively equivalent to $X_\B'$, which has degree $d'$.
It follows from construction that $C_{\pi,\ff}=C_{\Red(\pi)}$.

If $\ell=1$ and $\B \subset m\cdot \ZZ^2$, then clearly $d' \leq d/m$.
Suppose instead that $\ell>1$. 
Let $v,w\in\ZZ^{\ell+1}$ be the endpoints of the convex hull of $\B$. Since $\pi$ is a Cayley structure and $\ell>1$, there must be some index $i$ such that either $0<v_i<d$ or $0<w_i<d$. Without loss of generality, assume that $0<v_0<d$. Since $\pi$ is a Cayley structure and $w$ is the other endpoint of $\conv \B$, $w$ must have the form $(0, w_1, \ldots, w_\ell)$. Then $v+\langle \B \rangle$ intersects $\Delta_\ell(d)$ in at most $v_0+1$ lattice points, and we conclude that $d'\leq v_0<d$.
\end{proof}
\subsection{Cusps}
In this section, we assume that $\pi$ is primitive. Let $\ff$ be general, with corresponding roots $P_0,\ldots,P_\ell \in \PP^1$. We identify these points with their images on the rational normal curve $C_d$.
\begin{lemma}\label{lemma:markedcusp}
The center of projection $L$ meets the tangent line \[T_{P_i} C_d\] for some $i = 0, \ldots, \ell$ if and only if $\B$ is cuspidal.
\end{lemma}
\begin{proof}
After changing coordinates on $\PP^1$ and rescaling the $f_j$, we may assume without loss of generality that $P_i=(1:0)$, $f_i=y_1$,  and $f_j=y_0-a_jy_1$ for all $j\neq i$. We give $\PP^d$ the dual coordinates to $\{y_0^d, y_0^{d-1}y_1, \ldots, y_1^d\}$, so the tangent line
$T_{P_i} C_d$ is the span of $(1:0:\ldots:0)$ and $(0:1:0:\ldots:0)$. The center of projection $L$ meets $T_{P_i} C_d$ at the point $(\lambda:1:0 :\ldots:0)$ for some $\lambda \in \KK^*$ if and only if for every $v\in \B$, 
\[
	\ff^v=y_1^{v_i}\prod_{j\neq i} (y_0-a_jy_1)^{v_j}
\]
is orthogonal to $(\lambda,1,0,\ldots,0)$.
This is equivalent to requiring that for every $v\in \B$, $v_i\neq 1$, and if $v_i=0$, we have 
\[
\sum_j a_jv_j=\lambda.
\]
If two distinct elements $v, v'$ have $v_i = 0$, then for $L$ to meet $T_{P_i}C_d$ we must have
\[
\sum_j a_j(v_j-v_j')=0.
\]
This is impossible if $\ff$ is general, so if $L$ meets $T_{P_i}C_d$, we see (since $\pi$ is a Cayley structure) that there is a unique $v$ such that $v_i = 0$. Thus $\B$ is cuspidal.

Conversely, if $\B$ is cuspidal, then for some index $i$ and for every $v\in \B$, $v_i\neq 1$, and there is a unique $v$ such that $v_i = 0$. Since $\ff$ is general, for that $v$ we have $\sum_j a_jv_j\neq 0$ and so $L$ intersects $T_{P_i}$ at a point other than $P_i$.
\end{proof}

\begin{lemma}\label{lemma:unmarkedcusp}
The center of projection $L$ does not intersect any $T_{Q} C_d$ for $Q\notin \{P_0,\ldots,P_\ell\}$.
\end{lemma}
\begin{proof}
If $\ell=1$, the map $\rho_{\B,\ff}$ is a monomial map, and the only points of intersection with tangents can be $P_0$ and $P_1$.
We now assume $\ell>1$. Since we assume $\pi$ is primitive,  $\langle \B \rangle$ has rank at least two.

	Consider the moduli space	 $M_{0,\ell+2}$ of $\ell+2$ marked points $P_0,\ldots,P_\ell,Q$ on $\PP^1$.
	Let $V$ be the locus inside $M_{0,\ell+2}$ of those $P_0,\ldots,P_\ell,Q$ such that $L$ intersects the tangent line of $C_d$ at $Q$. Note that
	the center of projection $L$ depends on the points $P_0,\ldots,P_\ell$.
	Consider the map 
	\[
		\phi:\M_{0,\ell+2}\to \M_{0,\ell+1}
	\]
obtained by forgetting $Q$. We wish to show $\phi(V) \ne \M_{0,\ell+1}$, that is, for general $\ff$, the center of projection $L$ does not intersect any $T_{Q} C_d$ for $Q\notin \{P_0,\ldots,P_\ell\}$. It is enough to show
\[\dim \phi(V)<\ell-2=\dim \M_{0,\ell+1}.\]

Fix coordinates on $\PP^1$ so that $Q=(0:1)$ and $P_i=(1:z_i)$ with $z_0=0$ and $z_1=1$. Using a computation similar to in the proof of Lemma \ref{lemma:markedcusp}, it is straightforward to show that $V$ is cut out by the equations
\[
z\cdot v=0\qquad v\in \langle \B\rangle.
\]
In particular, the codimension of $V$ is the rank of $\langle \B \rangle$, which is at least two. Hence
$\dim V\leq \ell+2-3-2=\ell-3$, so $\dim \phi(V)<\ell-2$ as required.
\end{proof}

\subsection{Nodes}\label{sec:nodes}
We continue under the assumption that $\pi$ is primitive and $\ff$ is general.
\begin{lemma}\label{lemma:marked}
The center of projection $L$ meets a secant line through $P_i$ and $P_j$ if and only if 
there exists $v \in \pi(\tau)$ such that $v_i = v_j = 0$, and $v'_i, v'_j > 0$ for all $v' \ne v$.

Moreover, if $L$ meets a secant line passing through $P_i$, then that secant line must also pass through $P_j$ for some $j\neq i$.
\end{lemma}
\begin{proof}
If $\ell=1$, then the map $\rho_{\B,\ff}$ is a monomial map, and $L$ cannot intersect the secant line through $P_0$ and $P_1$. Likewise, the condition in the statement of the lemma cannot be fulfilled. Thus, we may assume going forward that $\ell>1$.

We first show that if $L$ meets a secant line passing through $P_i$ and some point $Q\in \PP^1$, then $Q = P_j$ for some $j$. Fix coordinates on $\PP^1$ so that $Q=(0:1)$ and $P_i=(1:0)$. Then the secant line through $P_i$ and $Q$ is the span of $(1:0:\cdots:0)$ and $(0:\cdots:0:1)$. Let $v\in \B$ be any element with $v_i>0$. Then $y_1$ divides $\ff^v$. The secant line must contain a non-zero vector orthogonal to $\ff^v$, so the coefficients of $y_0^d$ and $y_1^d$ in $\ff^v$ are proportional, hence both $0$. Then $y_0$ divides $\ff^v$ as well, and so $Q = P_j$ for some $j\neq i$.
We also see that, for any $v\in \B$, if $v_i>0$, then $v_j>0$ (and vice versa).

We now show the first claim of the lemma. Suppose that for all $v \in \B$, $v_i > 0$ if and only if $v_j > 0$. Let $\B_{ij} = \{v \in \B : v_i = v_j = 0\}$.
As above, we fix coordinates on $\PP^1$ so that $P_i=(1:0)$, $P_j=(0:1)$, and write $P_k=(z_k:1)$ for $k\neq i,j$ (where we fix $z_k=1$ for some $k\neq i,j$).

Consider the moduli space $M_{0,\ell+1}$ of $\ell+1$ marked points $P_0,\ldots,P_\ell$ on $\PP^1$. Let $V$ be the locus inside $M_{0,\ell+1}$ of those $P_0,\ldots,P_\ell$ such that $L$ intersects the secant line through $P_i$ and $P_j$. As above, we determine when $V \ne M_{0, \ell+1}$ by dimension counting. It is straightforward to see that $V$ is cut out by the equations
\begin{equation*}
z^v=1\qquad \forall v\in \langle \B_{ij} \rangle. 
\end{equation*}
This locus has codimension equal to the rank of $\langle \B_{ij} \rangle$. 
We conclude that for general $\ff$, $L$ intersects the secant through $P_i$ and $P_j$ if and only if $|\B_{ij}| = 1$.
\end{proof}

We next analyze the condition that $L$ meets an \emph{unmarked} secant line, i.e.~a secant line through points $Q_0,Q_1$ both distinct from $P_0,\ldots,P_\ell$.
We work over the moduli space $M_{0,\ell+3}$. Let $V \subseteq M_{0,\ell+3}$ be the locus of $\ell + 3$ marked points $P_0,\ldots,P_\ell,Q_0,Q_1 \in \mathbb{P}^1$ such that $L$ intersects the secant line through $Q_0$ and $Q_1$.
We choose coordinates on $\PP^1$ so that $Q_0=(1:0)$ and $Q_1=(1:0)$. Write $P_i=(z_i:1)$ for $i=0,\ldots,\ell$.
It is straightforward to check that $V$ is cut out by 
\begin{equation}\label{eqn:V}
z^v=1\qquad v\in \langle \B\rangle .
\end{equation}

\begin{lemma}\label{lemma:unmarked}
	If $\dim \B \geq 3$, $\dim \B=1$, or $e_i - e_j \in \langle \B \rangle$ for some $i\neq j$, then $L$ does not intersect any unmarked secant lines. If $\dim \B=2$, then the center of projection $L$ only intersects finitely many unmarked secant lines.
\end{lemma}
\begin{proof}
If $\dim \B=1$, then since $\pi$ is primitive, we have $\ell=1$ and the map $\rho_{\B,\ff}$ is a monomial map. The claim is easily verified in this case. Likewise, if $e_i - e_j \in \langle \B \rangle$ for some $i\neq j$, the locus $V$ is empty in $\M_{0,\ell+3}$, so $L$ cannot intersect an unmarked secant. We now assume $\dim \B \geq 2$ (and in particular $\ell\geq 2$).

	Consider the map $\phi : M_{0,\ell+3} \to M_{0,\ell+1}$ forgetting $Q_0$ and $Q_1$. We see that
	\[
\dim \phi(V) \leq \dim V=\ell-\dim \B.
	\]
	If $\dim \B \geq 3$, we see that $\dim \phi(V) < \ell - 2 = \dim \M_{0, \ell+1}$. If $\dim \B=2$ and the dimension of the image is $\ell-2$, then a generic fiber is zero-dimensional, that is, there are at most finitely many unmarked secant lines. If the dimension of the image is any smaller, the generic fiber is empty.
\end{proof}

\begin{proof}[Proof of Theorem \ref{thm:primitive}]
	We have seen in Lemma \ref{lemma:primitive} that if $\pi$ is imprimitive, $\deg C_{\pi,\ff}<d$. Conversely, by Lemma \ref{lemma:unmarked}, the center of projection $L$ intersects at most finitely many unmarked secant lines, and hence at most finitely many secant lines. It follows that the map $\rho_{\pi,\ff}$ is generically injective, hence $\deg C_{\pi,\ff}=d$.

	In the imprimitive case, $C_{\pi,\ff}=C_{\Red(\pi)}$ by construction. The Cayley structure $\Red(\pi)$ is primitive, and so the claim regarding its degree follows from the first statement of the theorem.
\end{proof}
\begin{proof}[Proof of Theorem \ref{thm:smooth} (start)]
	The claim regarding cuspidal singularities follows from Lemmas \ref{lemma:markedcusp} and \ref{lemma:unmarkedcusp}. The claim regarding $C_{\pi,\ff}$ being smooth will then follow from the claim on nodal singularities.

	If $\dim \B \geq 3$, $\dim \B = 1$, or $e_i - e_j \in \langle \B \rangle$ for some $i\neq j$, we see from Lemmas \ref{lemma:marked} and \ref{lemma:unmarked} that $C_{\pi,\ff}$ has a nodal singularity if and only if $\pi$ is nodal. It remains to deal with the case where $\dim \B=2$ and $e_i - e_j \notin \langle \B \rangle$ for all $i\neq j$. We address this case in the following section with Theorem \ref{thm:ranktwo}.
\end{proof}

\subsection{Nodes in Rank Two Case}\label{sec:nodes2}
We analyze whether $L$ intersects unmarked secant lines in the exceptional case: we assume, throughout this subsection, $\dim \B=2$ and $e_i - e_j \notin \langle \B \rangle$ for all $i\neq j$.

The center of projection $L$ intersects unmarked secant lines if and only if the image of $V$ in $\M_{0,\ell+1}$ has dimension $\ell-2$. Consider the solution set of \eqref{eqn:V} as a codimension-two subset $Y \subseteq \KK^{\ell+1}$. The condition that the image of $V$ in $\M_{0,\ell+1}$ has dimension $\ell-2$ is equivalent to the condition that the orbit of $Y$ under the natural rational $GL(2)$-action is dense in $(\KK^*)^{\ell+1}$. Here, $GL(2)$ acts by
\[
\begin{pmatrix}
	a_{11} & a_{12}\\
	a_{21} & a_{22}
\end{pmatrix} \cdot (z_i)=\left(\frac{a_{11}z_i+a_{12}}{a_{21}z_i+a_{22}}\right).
\]

\begin{lemma}\label{lemma:orbit}
The orbit of $Y$ under the natural rational $GL(2)$-action fails to be dense in $(\KK^*)^{\ell+1}$ if and only if for all $z\in Y$, the linear space generated by $z$ and $z^{-1}$ 
	is linearly dependent with $\langle \B\rangle ^\perp$.
\end{lemma}
\begin{proof}
We consider the image of the differential of the rational map
\[
GL(2)\times Y\dashrightarrow (\KK^*)^{\ell+1}
\]
at the point $(e,z)$ where $e$ is the identity element of $GL(2)$. 
Differentiating with respect to the $GL(2)$ directions, we obtain the span of 
\[z^2:=(z_i^2)_{i=0,\ldots,\ell},\qquad z:=(z_i)_{i=0,\ldots,\ell}\qquad \textrm{and}\qquad z^0:=(1)_{i=0,\ldots,\ell}.\]

On the other hand, the variety $Y\cap (\KK^*)^{\ell+1}$ is the quasitorus with character group $\ZZ^{\ell+1}/\langle \B \rangle$, and hence cocharacter lattice $\langle \B \rangle^\perp \cap {\ZZ^{\ell+1}}$. It follows that the tangent space of $Y$ at $z$ is the span of the vectors $z \langle \B \rangle^\perp$.
Since $\langle \B \rangle ^\perp$ contains $(1)_{i=0, \ldots, \ell}$, this tangent space contains $z$.

Putting this together, we obtain that the image of the differential at $(e,z)$ is the span of $z^2, z^0$ and $z\langle \B \rangle^\perp$. The orbit of $Y$ fails to be dense in $(\KK^*)^{\ell+1}$ if and only if for general (or equivalently all) $z\in Y$, this linear space has dimension less that $\ell+1=\rank \langle \B \rangle ^\perp+2$, in other words, the linear space generated by $z^2$ and $z^0$ is linearly dependent with $z\langle \B \rangle^\perp$. Dividing each coordinate by $z_i$, this is equivalent to requiring that  the linear space generated by $z$ and $z^{-1}$ is linearly dependent with $\langle \B \rangle^\perp$.
\end{proof}

We will now characterize lattices $\langle \B \rangle$ for which the condition of Lemma \ref{lemma:orbit} is fulfilled.
We note the following: if $I$ is an ideal generated by binomials in a polynomial ring, and $f \in I$ is a polynomial with support containing a monomial $z^v$, then there is a monomial $z^{v'}$ in the support of $f$, with $v' \ne v$, such that $z^v - z^{v'} \in I$. This can be shown, for example, by reducing $f$ using a Gr\"obner basis for $I$ consisting of binomials.

\begin{lemma}\label{lemma:rel}
Suppose that for all $z\in Y$, $z$ and $z^{-1}$ is linearly dependent with $\langle \B\rangle ^\perp$. Let $0\leq \alpha ,\beta\leq \ell$ with $\alpha\neq \beta$ and suppose further that the projection of $\langle \B \rangle$ to the $\alpha$th and $\beta$th coordinates has rank $2$.
Then $\langle \B \rangle$ contains an element of the form $e_\alpha-e_\beta+e_\gamma-e_\delta$ for some 
$\gamma\neq \delta$ with $(\alpha,\beta)\neq (\delta,\gamma)$
.
\end{lemma}
\begin{proof}
	Let $\langle \B\rangle_\QQ$ be the $\QQ$-span of $\langle \B\rangle$, and let $w_2,\ldots,w_{\ell}$ be a basis for $\langle \B\rangle_\QQ^\perp$. Let
	$W$ be the matrix with rows $z$, $z^{-1}$, $w_2,\ldots,w_\ell$, where $z$ is treated as a vector of indeterminates. We consider the Laurent polynomial $\det (W)$. Since for all $z\in (\KK^*)^{\ell+1}$ satisfying \eqref{eqn:V}, $z$ and $z^{-1}$ are linearly dependent with $\langle \B\rangle ^\perp$, and the ideal $I_{\langle\B\rangle}\subseteq \KK[z_i^{\pm 1}]$ corresponding to the equations \eqref{eqn:V} is radical, we have $\det(W)\in I_{\langle \B\rangle}$.  We may write 
\[
	\det(W)=\sum_{\alpha< \beta} \pm \Delta_{\alpha\beta} (z_\alpha z_\beta^{-1}-z_\beta z_\alpha^{-1})
\]
where $\Delta_{\alpha\beta}$ is the determinant of the matrix obtained from the matrix $W$ by deleting the first two rows and the $\alpha$th and $\beta$th columns. We note that $\Delta_{\alpha\beta}\ne0$ if and only if the projection of $\langle \B\rangle $ to the $\alpha$th and $\beta$th coordinates has rank two.

Suppose that $\Delta_{\alpha\beta}\neq 0$. Then the monomial $z_\alpha z_\beta^{-1}$ appears with non-zero coefficient in $\det(W)$. As discussed above, it follows that 
there must also exist $\gamma\neq \delta$ with $(\alpha,\beta)\neq (\delta,\gamma)$
such that $\Delta_{\gamma\delta}\neq 0$ and $z_\alpha z_\beta^{-1}-z_\delta z_\gamma^{-1}\in I_{\langle\B\rangle}$, or equivalently,
\[
e_\alpha-e_\beta+e_\gamma-e_\delta \in {\langle\B\rangle}. \qedhere
\]
\end{proof}
The element given above is sufficiently close to one of the form $e_i - e_j \notin \langle \B \rangle$ that we may directly characterize the possible lattices. Note that if $p (e_i - e_j) \in \langle \B \rangle$ where $p$ is a prime integer, then $p$ divides the index of $\langle \B \rangle$ in $\langle \B \rangle_\QQ \cap \mathbb{Z}^{\ell+1}$.

\begin{lemma}\label{lemma:saturated}
Suppose that $L$ does not intersect any unmarked secants, and $e_i-e_j\notin \langle \B\rangle$ for all $i\neq j$. 
The lattice $\langle \B \rangle$ is generated by two elements
\[
e_\alpha-e_\beta+e_\gamma-e_\delta \in {\langle\B\rangle}
\]
of the form promised by Lemma \ref{lemma:rel} whose supports intersect in at least one coordinate.
\end{lemma}
\begin{proof}
	Since $L$ does not intersect any unmarked secants, and there must exist some coordinates $\alpha,\beta$ such the projection of $\langle \B \rangle$ to these coordinates has rank two, it follows from Lemmas \ref{lemma:orbit} and \ref{lemma:rel} that there exists some element in $\langle \B \rangle$ of the form
\[
e_\alpha-e_\beta+e_\gamma-e_\delta.
\]
Since $e_i-e_j\notin{\langle\B\rangle}$, we cannot have $\alpha=\delta$ or $\beta=\gamma$. This means that we must be in one of the following cases:
\begin{itemize}
	\item $2e_\alpha-2e_\beta\in{\langle\B\rangle}$ (two-term relation);
	\item $2e_\alpha-e_\beta-e_\delta \in {\langle\B\rangle}$ ($\alpha,\beta,\delta$ distinct) or $2e_\beta-e_\alpha-e_\gamma\in {\langle\B\rangle}$ ($\alpha,\beta,\gamma$ distinct) (three-term relation);
	\item $e_\alpha-e_\beta+e_\gamma-e_\delta \in {\langle\B\rangle}$ ($\alpha,\beta,\gamma,\delta$ distinct) (four-term relation).
\end{itemize}

We analyze of each these cases in turn. Suppose first that $\langle \B \rangle$ has a {\bf two-term relation}, without loss of generality $2e_0-2e_1$.
Since $\ell \geq 2$, there must be an element $v$ of $\langle \B \rangle$ with $v_2\neq 0$, so we can apply Lemma \ref{lemma:rel} to $\alpha=0,\beta=2$ to obtain another two, three or four-term relation $w$ whose support overlaps with $2e_0-2e_1$.

Let $\Gamma$ be the lattice generated by $2e_0-2e_1$ and $w$. Clearly $\langle \B \rangle_\QQ=\Gamma_\QQ$, but we claim that in fact $\langle \B \rangle=\Gamma$. Indeed, if $w_i=1$ for some $i\geq 2$, the index of $\Gamma$ in $\Gamma_\QQ\cap \ZZ^{\ell+1}$ is $2$. However, $e_0-e_1\in\langle \B\rangle_\QQ \setminus \langle \B \rangle$, so $\langle \B \rangle$ also has index two in $\Gamma_\QQ\cap \ZZ^{\ell+1}$ and we are done. Only two possibilities remain: either $w=2e_1-2e_2$ or $w=-e_0-e_1+2e_2$. In both cases $\Gamma_\QQ \cap \ZZ^{\ell+1} = \langle e_0-e_1, e_1-e_2\rangle$, in which $\Gamma$ and $\langle \B \rangle$ must both then have index four, so they coincide.

We next suppose that $\langle \B \rangle$ has a {\bf three-term relation}, without loss of generality $2e_0-e_1-e_2$, but no two-term relations. If $\ell=2$, then since $\sum_i v_i=0$ for any $v\in\langle \B\rangle$, it follows that the projection of $\langle \B \rangle$ to the coordinates $1$ and $2$ must have full rank, so by Lemma \ref{lemma:rel} we may assume without loss of generality that $-e_0-e_1+2e_2\in \langle \B \rangle$. The sublattice generated by 
$2e_0-e_1-e_2$ and  $-e_0-e_1+2e_2$ has index $3$ in $\langle \B \rangle_\QQ\cap \ZZ^3$ and contains $3e_0 - 3e_2$. Since $e_0-e_2\notin \langle \B \rangle$ we conclude that it must coincide with $\langle \B \rangle$. 

If instead $\ell\geq 3$, then by Lemma \ref{lemma:rel} we must have a three or four-term relation $w$ whose support includes $2$ and $3$. If $w=-e_1-e_2+2e_3$, the sublattice generated by $2e_0-e_1-e_2$ and $w$ contains $2e_0 - 2e_3$, contradicting our assumption on $\langle \B \rangle$.
For all other $w$,  the sublattice generated by $2e_0-e_1-e_2$ and $w$ is saturated, 
so it must coincide with $\langle \B\rangle$.

Finally, suppose that $\langle \B \rangle$ only has {\bf four-term relations}. Without loss of generality, $e_0-e_1+e_2-e_3\in \langle \B \rangle$. If $\ell \geq 4$, Lemma \ref{lemma:rel} guarantees a four-term relation whose support includes $0$ and $4$; the lattice this generates is saturated, hence must coincide with $\langle \B \rangle$. If instead $\ell=3$, it follows from Lemma \ref{lemma:rel} that there must be another linearly independent four-term relation, without loss of generality $e_0+e_1-e_2-e_3$. But then we also obtain the two-term relation $2e_0-2e_3$, contradicting our assumption.
\end{proof}

\begin{lemma}\label{lemma:deg3}
Suppose that $L$ does not intersect any unmarked secants, and $e_i-e_j\notin \langle \B \rangle$ for all $i\neq j$. Then $\langle \B \rangle=\langle \B'\rangle$ for $\B'$ the image of a Cayley structure of degree at most three and with $\#\B'=3$.
\end{lemma}
\begin{proof}
	By Lemma \ref{lemma:saturated}, $\langle \B \rangle$ is generated by elements $v,w$ of the form stated in Lemma \ref{lemma:rel} whose support intersects in some index $i$.
After possibly scaling by $-1$, we assume that $v_i,w_i<0$.
	Define $u=(u_j)\in\ZZ^{\ell+1}$ as
	\[
		u_j=\max \{0,-v_j,-w_j\}. 
	\]
Since the sum of the negative entries of $v$ and $w$ are both $-2$, it follows that $d:=\sum_j u_j \leq 3$.
We set 
\[
	\B':=\{ u, u+v,u+w\}\subseteq \Delta_{\ell}(d).
\]
Then $\B'$ is the image of a Cayley structure of degree $d$. Indeed, since $\B$ is the image of a Cayley structure, for every $j$ either $u_j$, $(u+v)_j$, or $(u+w)_j\neq 0$. Likewise, by definition of $u$, for each $j$ either $u_j=0$, $u_j+v_j=0$, or $u_j+w_j=0$.
\end{proof}

\begin{lemma}\label{lemma:nomarked}
Suppose that $\#\B=3$ and $\B$ is the image of a Cayley structure of degree at most three. Assume that $e_i-e_j\notin \langle \B \rangle$ for any $i\neq j$. Then $L$ does not intersect any marked secants, or any tangents.
\end{lemma}
\begin{proof}
	Let $\B=\{u,v,w\}$.
	By Lemma \ref{lemma:unmarkedcusp}, $L$ does not intersect any unmarked tangents. Suppose $L$ intersects the tangent $T_{P_k} C_d$. Then without loss of generality, by Lemma \ref{lemma:markedcusp} $u_k=0$ and $v_k,w_k\geq 2$. But since $\sum v_j=\sum w_j \leq 3$, this implies $v-w=e_\alpha-e_\beta\in\B$ for some $\alpha,\beta$, a contradiction. 

	Suppose instead that $L$ intersects the marked secant through $P_i$ and $P_j$. Then without loss of generality, by Lemma \ref{lemma:marked} $u_i=u_j=0$ and $v_i,v_j,w_i,w_j>0$. Again since $\sum v_j=\sum w_j \leq 3$, this implies $v-w=e_\alpha-e_\beta\in\B$ for some $\alpha,\beta$, a contradiction. 
\end{proof}

\begin{lemma}\label{lemma:conic}
	Suppose that $\dim \B=2$, $\#\B=3$, and $\B$ is the image of a Cayley structure of degree $2$. Then up to permutation of the coordinates, $\langle \B\rangle$ is one of the lattices from Table \ref{table:secant}.
\end{lemma}
\begin{proof}
Since the degree is two and $\#\B=3$, we obtain $\ell \leq 5$. The claim follows from a straightforward case-by-case analysis.
\end{proof}

\begin{thm}\label{thm:ranktwo}
	Suppose that $\dim \B=2$ and $e_i-e_j\notin \langle \B \rangle$ for all $i\neq j$. Then $L$ does not intersect an unmarked secant if and only if $\langle \B \rangle$ is one of the lattices listed in Table \ref{table:secant}.
\end{thm}

\begin{proof}
	The property of $L$ intersecting unmarked secant lines depends only on $\langle \B\rangle $, not on $\B$. If $\langle \B \rangle$ is one of the lattices from Lemma \ref{lemma:conic} listed in Table \ref{table:secant}, then we may replace $\B$ by the image of a Cayley structure with $d=2$ and $\#\B=3$. Then $C_{\B,\ff}$ is a smooth (plane) conic, so $L$ cannot intersect any secant lines at all.
	
	Suppose we are not in one of these cases. Then by Lemma \ref{lemma:deg3}, we may replace 
	$\B$ by the image of a Cayley structure with $d=3$ and $\#\B=3$. Then $C_{\B,\ff}$ is a degree three rational plane curve, hence singular. By Lemma \ref{lemma:nomarked}, $L$ does not intersect any tangents or marked secants.
	Hence, in order for $C_{\B,\ff}$ to be singular, $L$ must intersect an unmarked secant.
\end{proof}

\section{Stratification of the Hilbert Scheme}\label{sec:hilb}
\subsection{Map to the Hilbert Scheme}
As noted in the introduction,
given a polynomial $P(m)\in\QQ[m]$ and a projective variety $X\subset \PP^n$, we let 
	$\Hilb_{P(m)}(X)$
	denote the \emph{Hilbert scheme} parametrizing closed subschemes of $X$ with Hilbert polynomial $P(t)$. In particular, $\Hilb_{d\cdot m +1}(X_\A)$ is the fine moduli space parametrizing closed one-dimensional subschemes of the toric variety $X_\A$ with the same Hilbert polynomial as a smooth degree $d$ rational curve.

	Fix a smooth degree $d$ Cayley structure $\pi$.
	By Theorems \ref{thm:primitive} and \ref{thm:smooth}, for a general point $\bp\in M_{0,\ell+1}$ and any $t\in T_\tau$, the rational curve
	$t\cdot C_{\pi,\ff(\bp)}$ has degree $d$ and is smooth, hence corresponds to a point 
	$[t\cdot C_{\pi,\ff(\bp)}]\in \Hilb_{d\cdot m+1}(X_\A)$.
	We thus obtain a rational map, when $\ell > 1$,
	\[
		M_{0,\ell+1}\times T_\tau \dashrightarrow \Hilb_{dm+1}.
	\]
	When $\ell = 1$, we instead have $T_\tau/T_\pi \dashrightarrow \Hilb_{dm+1}$.
	
Let $Z_\pi^\circ$ denote the image of this map, and $Z_\pi$ its closure in $\Hilb_{dm+1}(X_\A)$.
By Proposition \ref{prop:curvetocayley}, every degree $d$ smooth rational curve in $X_\A$ corresponds to a point in some $Z_\pi^\circ$. By Remark \ref{rem:le1} and Proposition \ref{prop:stabilizer-ell-ge-1}, the above rational maps have finite fibers. In particular, $\dim Z_\pi = \ell-2 + \dim \tau.$

\subsection{Partial Order on Cayley Structures}\label{sec:order}
We now define a partial order on the set of all degree $d$ smooth Cayley structures defined on some face of $\A$.

Given a map
\[
	\phi:\{0,\ldots,\ell\}\to \{0,\ldots,\ell'\},
\]
we obtain a linear map $\ZZ^{\ell+1}\to \ZZ^{\ell'+1}$ sending $e_i$ to $e_{\phi(i)}$. We also denote this linear map by $\phi$. Note that $\phi(\Delta_\ell(d)) \subseteq \Delta_{\ell'}(d)$.
\begin{defn}\label{defn:order}
Let $\tau,\tau'$ be faces of $\A$. Consider degree $d$ Cayley structures $\pi:\tau\to \Delta_\ell(d)$ and 
$\pi':\tau'\to \Delta_{\ell'}(d)$. We say that $\pi'\leq \pi$ if $\tau'$ is a face of $\tau$ and there exists a map
\[
	\phi:\{0,\ldots,\ell\}\to \{0,\ldots,\ell'\}
\]
such that $\pi'=\phi\circ \pi|_{\tau'}$.
\end{defn}
Note that $\phi$ is necessarily surjective, or else $\pi'(\tau')$ would lie in a face of $\Delta_{\ell'}(d)$. Clearly $\pi$ and $\pi'$ are equivalent Cayley structures if and only if $\pi\leq \pi'\leq \pi$.

	\begin{thm}\label{thm:smoothcomp}
		Let $\pi$ and $\pi'$ be smooth Cayley structures on faces of $\A$. Then $\pi'\leq \pi$ if and only if $Z_{\pi'}\subseteq Z_{\pi}$. 
\end{thm}
We will prove this theorem in the following subsection.
    \begin{cor}\label{cor:main}
	    The map $\pi \mapsto Z_\pi$ induces a bijection between equivalence classes of maximal smooth primitive degree $d$ Cayley structures and irreducible components of $\Hilb_{dm+1}(X_\A)$ whose general element is a smooth rational curve.
    \end{cor}
\begin{proof}
	Consider any point $z$ of $\Hilb_{dm+1}(X_\A)$ corresponding to a smooth curve $C$. Then $C$ is rational of degree $d$. By Proposition \ref{prop:curvetocayley}, $C$ is a torus translate of some curve $C_{\pi,\ff(\bp)}$, where $\pi$ is a Cayley structure. It follows from Theorems \ref{thm:primitive} and \ref{thm:smooth} that $\pi$ is smooth (and thus primitive). Thus, the point $z$ is contained in some $Z_\pi$.

	The irreducible components of $\Hilb_{dm+1}(X_\A)$ thus correspond to the maximal elements of $\{Z_{\pi}\}$, as $\pi$ ranges over smooth degree $d$ Cayley structures. But by Theorem \ref{thm:smoothcomp}, these maximal elements are given by exactly those $Z_\pi$ where $\pi$ is maximal among all smooth degree $d$ Cayley structures. Since $Z_\pi=Z_{\pi'}$ if and only if $\pi$ and $\pi'$ are equivalent, we obtain the desired bijection.
\end{proof}

\begin{ex}\label{ex:simplex}
Let $\A = \Delta_\ell(1)$, the unimodular simplex. Then up to equivalence $\A$ has a unique maximal Cayley structure of degree $d$, namely $\pi : \A \to \Delta_{(\ell+1)d-1}(d)$ given by
\[e_j \mapsto e_{jd} + e_{jd+1} + \cdots + e_{jd+d-1} \qquad j \in \{0, \ldots, \ell\}. \qedhere\]
\end{ex}

\begin{ex}\label{ex:fanohilb}
	We continue with the set $\A$ from Example \ref{ex:fano} and pictured in Figure \ref{fig:introex}. As in Example \ref{ex:cayley}, consider the face 
	$\tau=\{(1,0,0),(0,-1,0),(0,0,1)\}$. We may consider the length $1$ degree $2$ Cayley structure $\pi'''$ on $\tau$ sending $(1,0,0)$ to $e_0+e_1$, $(0,-1,0)$ to $2e_1$, and $(0,0,1)$ to $2e_0$. This is pictured in Figure \ref{fig:cayley2}.

	The Cayley structure $\pi'''$ is not maximal: letting $\pi'$ and $\pi''$ be as in Example \ref{ex:cayley}, we have $\pi''' \leq \pi'$ and $\pi'''\leq  \pi''$. Indeed, we obtain $\pi'''$ from $\pi'$ by restricting it from $\A$ to the face $\tau$. On the other hand, we obtain $\pi'''$ from $\pi''$ by composing $\pi''$ with the map sending $e_0,e_1,e_2$ to $e_0$ and  $e_3,e_4,e_5$ to $e_1$.
	By Theorem \ref{thm:smoothcomp}, we see that the conics corresponding to the Cayley structure $\pi'''$ may be deformed into two different kinds of conics: those corresponding to $\pi'$ (certain conics intersecting the open orbit of $X_\A$ that meet the boundary in exactly two points) and those corresponding to $\pi''$ (conics contained in a boundary stratum of $X_\A$ isomorphic to $\PP^2$ that meet the boundary of this $\PP^2$ in six points).

	It is straightforward to determine all the maximal Cayley structures (up to equivalence) on $\A$ and its faces. Because of the geometry of $\A$, any Cayley structure defined on all of $\A$ must have length one; the only possibilities are exactly the nine described in Example \ref{ex:fano}. 
	By Corollary \ref{cor:main}, this gives us nine components of the Hilbert scheme of conics. Each of these components has dimension
	\[
\ell-2+\dim \A=2.
	\]

	On the other hand, any facet of $\A$ is a unimodular $2$-simplex, which has a unique maximal Cayley structure (cf. Example \ref{ex:simplex}). The resulting $12$ components of the Hilbert scheme of conics all have dimension
	\[
\ell-2+\dim \tau=5-2+2=5.
	\]
	In fact, each of these components is parametrizing conics in one of the $\PP^2$-boundary strata, so it is just a copy of $\PP^5$.
\end{ex}
\begin{figure}\tiny{
		\begin{center}
\begin{tikzpicture}[scale=1.2]
\draw[fill] (0,-1) circle [radius=0.04] node[below right] {$2e_1$};
\draw[fill] (1,0) circle [radius=0.04] node[below right] {$e_0+e_1$};
\draw[fill] (.5,1.3) circle [radius=0.04] node[above] {$2e_0$};
\draw[fill,lightgray] (1,0,0) -- (0,-1,0) -- (.5,1.3) ;
\draw (0,1) -- (-1,0) -- (-1,-1) -- (0,-1) -- (1,0) -- (1,1) -- (.5,1.3) -- (1,0) -- (.5,1.3) -- (0,-1) -- (.5,1.3) -- (-1,-1) -- (.5,1.3) -- (-1,0) -- (.5,1.3) -- (0,1);
\draw (-1,-1) -- (-.5,-1.3) -- (0,-1);
\draw[dashed] (0,1) -- (1,1) -- (-.5,-1.3);
\draw[dashed] (1,0) -- (-.5,-1.3);
\draw[dashed] (-1,0) -- (-.5,-1.3);
\draw[dashed] (0,1) -- (-.5,-1.3);
\end{tikzpicture}
\end{center}}
\caption{A non-maximal Cayley structure (Example \ref{ex:fanohilb})}\label{fig:cayley2}
\end{figure}
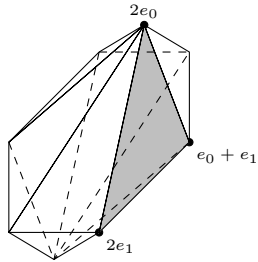

The set of Cayley structures on $\tau$ of degree $d$ is always finite, since the length $\ell$ is bounded, e.g. $\ell < |\tau| \cdot d$. As such it is possible to find the maximal Cayley structures on $\A$ by enumerating all possible Cayley structures. It would be interesting to have a direct characterization of maximality:

\begin{question}
Is there a combinatorial criterion for maximality of a Cayley structure $\pi : \tau \to \Delta_\ell(d)$ on $\A$?
\end{question}

\subsection{Smooth Degenerations}
\begin{proof}[{Proof of Theorem \ref{thm:smoothcomp}}]
	We first show that if $\pi'\leq \pi$, then $Z_{\pi'}\subseteq Z_\pi$. Let $\phi$ be as in Definition \ref{defn:order}.
	Consider some $\ff'$ such that $C_{\pi',\ff'}$ is smooth of degree $d$. Let $v\in \Hom(M,\ZZ)$ be such that the face of $\tau$ on which $v$ is minimal is exactly $\tau'$. We may view $v$ as a one-parameter subgroup of $T=\spec \KK[M]$, giving us a map $v:\KK^*\to T$. Let $g_{0},\ldots,g_{\ell}\in\KK[y_0,y_1]$ be general linear forms.
For $t\in \Aff^1$, set 
\[
	\ff(t):=(f'_{\phi(0)}+t\cdot g_{0},\ldots,f'_{\phi(\ell)}+t\cdot g_{\ell}).
\]
	We thus obtain a morphism
	\begin{align*}
		\Aff^1&\to Z_{\pi}\\
		t&\mapsto [v(t)\cdot C_{\pi,\ff(t)}].
	\end{align*}
	At the point $t=0$, the corresponding curve is exactly $C_{\pi',\ff'}$. Hence, $[C_{\pi',\ff'}]$ is in the closure of $Z_\pi$. Since $Z_\pi$ is $T$-invariant, it follows that all of $Z_{\pi'}$ is contained in $Z_\pi$.

	We now show instead that $Z_{\pi'}\subseteq Z_\pi$ implies that $\pi'\leq \pi$. 
		We first note that if $Z_{\pi'}$ is in the closure of $Z_{\pi}$, clearly $\tau'$ must be contained in $\tau$, and hence a face of it. Now, take a curve $Y$ in $\Hilb_{dm+1}$ passing through a general point of $Z_{\pi}$ and a general point of $Z_{\pi'}$, say at $\eta\in Y$.
After pulling back along an \'etale map, we may assume that the family over $Y$ is a trivial family of rational curves over an irreducible affine curve $Y=\spec R$, see e.g. \cite[Exercise 25.2]{deftheory}.
Thus, we have a map  
\[
\PP^1\times Y\to X_\tau\times Y
\]
given by degree $d$ forms $F_u\in R[y_0,y_1]_d$ for each $u\in \tau$. After possibly further pullback along a finite map, we may assume that each $F_u$ factors as a product of linear factors. 

From Proposition \ref{prop:curvetocayley}, the factors of the $\{F_u\}$ at a general point determine a Cayley structure $\tau\to \Delta_{\ell+1}$, which is exactly the Cayley structure $\pi$. In particular, the $\{F_u\}$ have $\ell+1$ distinct factors $f_0,\ldots,f_\ell\in R[y_0,y_1]_1$ up to scaling by non-zero elements of the field of fractions of $R$.
Without loss of generality, we will order them so that the factors of the $F_u$ for $u\in \tau'$ are $f_0, \ldots, f_j$ for some $j$.

Similarly, the factors of the $\{F_u(\eta)\}$ for $u\in\tau'$ determine a Cayley structure $\tau'\to \Delta_{\ell'+1}$ (and in particular, $F_u(\eta)=0$ for $u\notin \tau'$). After enumerating the factors of the $\{F_u(\eta)\}_{u\in\tau'}$ up to scaling, the map $f_i\mapsto f_i(\eta)$ for $i\leq j$ induces a surjective map $\{0,\ldots,j\}\to \{0,\ldots,\ell'\}$. We may extend this arbitrarily to a surjective map $\phi:\{0,\ldots,\ell\}\to \{0,\ldots,\ell'\}$. 
It follows from the construction of $\phi$ that $\pi'=\phi\circ\pi$, and hence $\pi'\leq \pi$.
\end{proof}

\section{Torus Orbits}\label{sec:orbits}
\subsection{Limiting Cycles}
Let $\Chow_d(X_\A)$ denote the \emph{Chow variety} parametrizing degree $d$ one-cycles in $X_\A$. For a curve $C\subseteq X_\A$, we denote the corresponding
one-cycle by $\{C\}$.
Given a primitive Cayley structure $\pi$ of degree $d$, we wish to describe the closure in $\Chow_d(X_\A)$  
of the torus orbit
\[
	T\cdot\{C_{\pi,\ff}\}
\]
for any choice of $\ff$ with whose entries have distinct roots. This orbit closure is a (potentially non-normal) complete toric variety.
In order to understand this orbit closure, we will describe the limit of $\{C_{\pi,\ff}\}$ under a one-parameter subgroup.

We first introduce a bit of notation. We refer the reader to \cite[\S3.1]{cls} for the correspondence between normal toric varieties and fans.

Let $\tau$ be a face of $\A$, let $V=\Hom(\langle \tau \rangle,\RR)$ and consider any $v\in V$. We denote by $\tau^v\prec \tau$ the face on which $v$ is minimal.\footnote{To make this precise, we may view $v$ as a linear function on $\tau$ by fixing an element $u'\in \tau$ and defining the value of $v$ on $u\in \tau$ to be $v(u-u')$. The function $v$ depends on the choice $u'$, but only up to a constant term, and hence $\tau^v$ is well-defined.}
	The dependence of $\tau^v$ on $v$ can be understood combinatorially. Let $\Sigma$ be the inner normal fan of $\tau$ viewed as a subset of $V$. 
Then $\tau^v=\tau^{v'}$ if and only if $v$ and $v'$ belong to the relative interior of the same cone of $\Sigma$. Accordingly, for any cone $\sigma\in \Sigma$, we denote by $\tau^\sigma$ the face $\tau^v$ where $v$ is any element in the relative interior of $\sigma$.

Let $\pi:\tau\to \Delta_\ell(d)$ be a primitive Cayley structure and let $0 \leq i \leq \ell$.
\begin{defn}\label{defn:iface}
	A face $\tau'\preceq \tau$ is an \emph{$i$-face} if $e_i^*$ is non-constant on $\pi(\tau')$. Equivalently, $\pi^*(e_i^*) \in V$ is not in the linear span of the cone corresponding to $\tau'$.
\end{defn}
\noindent We denote by $\Sigma_i$ the subfan of $\Sigma$ consisting of the cones corresponding to $i$-faces. Finally, for $v\in V$ let $\Sigma_i^v$ consist of those cones in $\Sigma_i$ such that the ray $v+\RR_{>0} \pi^*(e_i^*)$ intersects their relative interior.

\begin{rem} \label{rem:minimal-i-v-face}
A minimal $i$-face is always of dimension $1$ (if $e_i^*$ is nonconstant on $\pi(\tau')$, it is nonconstant on some edge of $\tau'$). Equivalently, the maximal cones of $\Sigma_i$ are of codimension $1$. Accordingly, for general $v \in V$, $\Sigma_i^v$ consists only of cones of codimension $1$ and corresponds only to edges of $\tau$. Note that $\Sigma^i_v$ is not a fan, only a collection of cones.
\end{rem}

We will use the restrictions of $\pi$ to $\tau^v$ and to each cone of $\Sigma_i^v$ to describe the limit of $\{C_{\pi, \ff}\}$, as follows.

For any $i=0,\ldots,\ell$, let $\kappa_i:\ZZ^{\ell+1}\to\ZZ^2$ be the linear map defined by
\[
	\kappa_i(e_j)=\begin{cases}
e_0& j=i\\
e_1& j\neq i
	\end{cases}.
\]
Let $v\in \Hom(\langle \tau \rangle,\ZZ)$. First, if $\pi|_{\tau^v}$ is non-constant, let $m^v$ be the multiplicity of $\Res(\pi|_{\tau^v})$ and let
\[\pi^v=\Red(\Res(\pi|_{\tau^v}))\]
be a basepoint-free weak Cayley structure.

For any basepoint $i$ of $\pi|_{\tau^v}$ and any $\sigma\in\Sigma_i$, let $m_i^\sigma$ be the multiplicity of $\Res(\kappa_i\circ \pi|_{\tau^\sigma}))$ and let
\[\pi_i^\sigma :=\Red(\Res(\kappa_i\circ \pi|_{\tau^\sigma})) : \tau^\sigma \to \Delta_1(d/m_i^\sigma).\]
Then $\pi_i^\sigma$ is a Cayley structure of length one.

Finally, for linear forms $f, g$, set
\[
f\wedge g=f(1,0)g(0,1)-g(1,0)f(0,1).
\]
Note that $f \wedge g = 0$ if and only if $f$ and $g$ have a common root. Viewing $\pi^*(e_j^*) \in N_\tau \cong \Hom(\KK^*, T_\tau)$, we define
\[
t_i = \prod_j \pi^*(e_j^*)(f_j \wedge f_i).
\]

\begin{thm}\label{thm:limit}
Let $\pi:\tau\to\Delta_\ell(d)$ be a primitive degree $d$ Cayley structure.
	For general choice of $\ff$,
	the limit under $v\in N_\tau$ of the degree $d$ cycle $\{C_{\pi,\ff}\}$ is
\begin{align*}
	m^v\cdot \{C_{\pi^v,\ff}\}+\sum_{i,\sigma\in\Sigma_i^v} m_i^\sigma\cdot \{t_i\cdot C_{\pi_i^\sigma}\}\qquad&\textrm{if }\pi|_{\tau^v}\textrm{ non-constant;}\\
	\sum_{i,\sigma\in\Sigma_i^v} m_i^\sigma\cdot \{t_i\cdot C_{\pi_i^\sigma}\}\qquad&\textrm{if }\pi|_{\tau^v}\textrm{constant}.
\end{align*}
The sum is taken over all basepoints $i$ of $\pi|_{\tau^v}$.
Furthermore, if we assume that $\pi|_{\tau^v}$ is constant, the claim is true for any $\ff$ whose entries have distinct roots such that $\deg C_{\pi,\ff}=d$.
\end{thm}
We will prove this theorem in \S\ref{sec:blowup}.

\begin{ex}\label{ex:limit}
	We apply Theorem \ref{thm:limit} to compute the limit of $C_{\pi,\ff}$ under the one-parameter subgroup $v=(-1,-1,-1)$, where $\pi$ is the Cayley structure from Example \ref{ex:cayley}.
	The face $\tau^v$ is just the point $(1,1,0)$, so $\pi|_{\tau^v}=2e_0$ is constant and has $0$ as a basepoint. Note that $\pi^*(e_0^*)=(0,1,0)$.

	The cones of $\Sigma_0^v$ are exactly the rays generated by $(-1,0,-1)$ and $(-1,1,-1)$; the corresponding faces of $\A$ are the convex hulls of 
	$\{(1,1,0),(0,0,1),(1,0,0)\}$ and $\{(0,0,1),(1,0,0),(0,-1,0)\}$. These faces, along with the images of the corresponding Cayley structures $\pi_i^\sigma$ are pictured in Figure \ref{fig:limit}. Note that for the first face we have permuted the roles of $e_0$ and $e_1$ for convenience. Both of these Cayley structures are primitive of degree one, hence all $m_i^\sigma=1$. Furthermore, since $\pi$ already had length $1$, $t_0=1$. 

	We see that the limit of $C_{\pi,\ff}$ is the union of two distinct lines intersecting in a point. Each of these lines is contained in one of the $\PP^2$ boundary strata depicted in Figure \ref{fig:limit}; the intersection of these lines is not a torus fixed point.
\end{ex}

\begin{figure}\tiny{
		\begin{center}
\begin{tikzpicture}[scale=1.2]
\draw[fill] (0,-1) circle [radius=0.04] node[below right] {$e_1$};
\draw[fill] (1,1) circle [radius=0.04] node[above right] {$e_1$};
\draw[fill] (1,0) circle [radius=0.04] node[below right] {$e_0$};
\draw[fill] (.5,1.3) circle [radius=0.04] node[above] {$e_0$};
\draw[blue,pattern=north west lines,pattern color=blue] (1,0,0) -- (1,1,0) -- (.5,1.3) ;
\draw[blue,pattern=horizontal lines,pattern color=red] (1,0,0) -- (0,-1,0) -- (.5,1.3) ;
\draw (0,1) -- (-1,0) -- (-1,-1) -- (0,-1) -- (1,0) -- (1,1) -- (.5,1.3) -- (1,0) -- (.5,1.3) -- (0,-1) -- (.5,1.3) -- (-1,-1) -- (.5,1.3) -- (-1,0) -- (.5,1.3) -- (0,1);
\draw (-1,-1) -- (-.5,-1.3) -- (0,-1);
\draw[dashed] (0,1) -- (1,1) -- (-.5,-1.3);
\draw[dashed] (1,0) -- (-.5,-1.3);
\draw[dashed] (-1,0) -- (-.5,-1.3);
\draw[dashed] (0,1) -- (-.5,-1.3);
\end{tikzpicture}
\end{center}}
\caption{Faces for limiting Cayley structures (Example \ref{ex:limit}).}\label{fig:limit}
\end{figure}
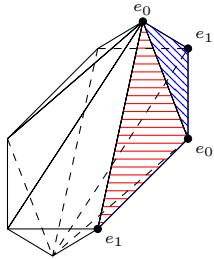

\subsection{Orbits in $\Chow_d(X_\A)$}\label{sec:orbit}
We will now describe the closure of
\[
	T\cdot \{C_{\pi,\ff}\}
\]
in $\Chow_d(X_\A)$.

Let $\varepsilon$ be a minimal $i$-face of $\tau$. (Recall that $\dim \varepsilon = 1$, cf. Remark \ref{rem:minimal-i-v-face}). We define its $i$-multiplicity to be 
\begin{align*}
	\mult_i(\varepsilon):=\max_{u,u'\in \varepsilon} (\pi(u)_i-\pi(u')_i)/L(\varepsilon),
\end{align*}
where $L(\varepsilon)$ is the length of $\varepsilon$ with respect to the lattice $\langle \varepsilon \rangle$. This is, equivalently, the multiplicity of $\Res(\kappa_i \circ \pi|_\varepsilon)$. If $\varepsilon$ is not minimal, we set $\mult_i(\varepsilon)=0$.

Let $Z(\tau)$ be the group of formal integral linear combinations of faces of $\tau$.
Define a map
\begin{align*}
\phi:V&\to Z(\tau)\\
\phi(v)&=\sum_{\substack{i,\sigma\in \Sigma_i\\
	v\in \sigma-\RR_{\geq 0} \pi^*(e_i^*)}} \mult_i(\tau^\sigma)\cdot \tau^\sigma.
\end{align*}
Thus, $\tau^\sigma$ contributes to $\phi(v)$ if the ray $v + \RR_{\geq 0} \pi^*(e_i^*)$ intersects $\sigma$ and $\mult_i(\tau^\sigma) \ne 0$ (i.e. $\tau^\sigma$ is a minimal $i$-face).
\begin{defn}\label{defn:sigmapi}
Let $\Sigma_\pi$ consist of the closures in $V$ of those full-dimensional regions on which $\phi$ is constant, along with all faces of these sets.
\end{defn}

Recall the definition of the lattice $N_{\pi,\ff}$ along with the map $N_\tau\to N_{\pi,\ff}$ from Definition \ref{defn:stablattice}.

\begin{thm}\label{thm:fan}
	Let $\pi$ be a primitive degree $d$ Cayley structure. Let $\ff$ have distinct roots and be such that $C_{\pi,\ff}$ has degree $d$. Let $Z$ be the normalization of
	\[\overline{T\cdot \{C_{\pi,\ff}\}}\subseteq \Chow_d(X_\A).\]
\begin{enumerate}
	\item If $\ell>1$, the set $\Sigma_\pi$ is a fan. If $\ell=1$, the set $\Sigma_\pi$ modulo its one-dimensional lineality space is a fan; the lineality space is generated by the image of $\pi^*$.
	\item If $\ell>1$ and $\ff$ is sufficiently general, then $Z$ is the toric variety associated to the fan $\Sigma_\pi$ with respect to the lattice $N_\tau$.
	\item More generally, for $\ell$ arbitrary and any $\ff$ satisfying the above hypotheses, $Z$ is the toric variety associated to the image of $\Sigma_\pi$ in $N_{\pi,\ff}\otimes \RR$ with respect to the lattice $N_{\pi,\ff}$. 
\end{enumerate}
\end{thm}

\begin{proof}
	The variety $Z$ is a complete toric variety. The action by the torus $T$ on $Z$ factors through the quotient torus $T_\tau$. The kernel of this action is exactly the stabilizer of $\{C_{\pi,\ff}\}$ which is described in Proposition \ref{prop:stabilizer-ell-ge-1}. The quotient of $T_\tau$ by this kernel is the torus whose character lattice is $N_{\pi,\ff}$; this is the lattice whose associated vector space contains the fan $\Sigma'$ for $Z$. We see that the second claim of the theorem will follow from the third, since for $\ff$ general and $\ell>1$, $N_\tau=N_{\pi,\ff}$, see Corollary \ref{cor:stab}.

	In Lemma \ref{lemma:limit} below we will show that for general $v,v'\in N_\tau$, the limits of $\{C_{\pi,\ff}\}$ under the corresponding one-parameter subgroups of $T_\tau$ coincide if and only if $v$ and $v'$ belong to the same (full-dimensional) region in $\Sigma_\pi$. 
	On the other hand, general $w,w'\in N_{\pi,\ff}$ belong to the same (full-dimensional) cone of $\Sigma'$ if and only if they give the same limit of 
	$\{C_{\pi,\ff}\}$, see e.g. \cite[Proposition 3.2.2]{cls}. Since any complete fan is determined by its full-dimensional cones, it follows that $\Sigma_\pi$ is the preimage in $V$ of $\Sigma'$. The third claim now follows. 

Recalling the discussion following Definition \ref{defn:stablattice}, the map $N_\tau\to N_{\pi,\ff}$ has finite kernel when $\ell>1$, and kernel generated by the image of $\pi^*$ when $\ell=1$. The first claim now follows as well.
\end{proof}
\begin{lemma}\label{lemma:limit}
	For general $v\in N_\tau$, the limit of $\{C_{\pi,\ff}\}$ under $v$ is the composition of $\phi(v)$ with the map sending a face $\varepsilon\prec \tau$ to the cycle $\{X_{\varepsilon}\}$. In particular, for generic $v,v'\in N_\tau$, the limits of $\{C_{\pi,\ff}\}$ coincide if and only if $v$ and $v'$ belong to the same full-dimensional region of $\Sigma_\pi$.
\end{lemma}

\begin{proof}
Fix a general $v\in N_\tau$. By Theorem \ref{thm:limit}, the limit of $\{C_{\pi,\ff}\}$ is
\begin{equation}\label{eqn:limit}
		\sum_{i,\sigma\in\Sigma_i^v} m_i^\sigma\cdot \{t_i\cdot C_{\pi_i^\sigma}\}.
	\end{equation}
	Indeed, since $v$ is general, it follows that $\tau^v$ is a vertex, so $\pi_{\tau^v}$ is constant and there is no $m^v\cdot \{C_{\pi^v,\ff}\}$ term in the limit.

	We claim that \eqref{eqn:limit} agrees with 
\[		
\sum_{i,\sigma\in\Sigma_i^v} \mult_i(\tau^\sigma)\cdot \{X_{\tau^\sigma}\}.
\]
Indeed, since $v$ is general, every $\sigma\in \Sigma_i^v$ has codimension one, so every corresponding $\tau^\sigma$ is an edge of $\tau$. It follows that $C_{\pi_i^\sigma}$ is torus fixed, so we may dispense with the action by the $t_i$. Moreover, for any $i$, the Cayley structure $\pi_i^\sigma$ results in the curve $X_{\tau^\sigma}$ .
Finally, it is clear from the definition that $\mult_i(\sigma)$ is just the multiplicity of $\Res(\kappa_i \circ \pi|_{\tau^\sigma})$. This shows the first claim. The second is immediate.
\end{proof}

\begin{ex}\label{ex:fanochow}
	We continue our analysis of the Cayley structures $\pi$ and $\pi'$ from Example \ref{ex:cayley}. For ease of referring to the edges of $\A$, we label the vertices in the left of Figure \ref{fig:label}. The normal fan of the convex hull of $\A$ is pictured in Figure \ref{fig:introex}. For the Cayley structure $\pi$, the behaviour of the map $\phi$ is described in Figure \ref{fig:phiv1} as we now explain. Depicted in the figures are the $z=1$ and $z=-1$ hyperplanes in $V\cong\RR^3$. Since $\phi$ is invariant under scaling, the values of $\phi$ on these two slices determines $\phi$ everywhere, except for on the $z=0$ hyperplane.

Since $\pi$ has length $1$, contributions to $\phi$ come from $0$-basepoints (and $\sigma\in \Sigma_0$) or $1$-basepoints (and $\sigma\in \Sigma_1$).
We have subdivided the $z=1$ and $z=-1$ slices into full-dimensional polyhedra on whose interiors the $0$-basepoint and $1$-basepoint contributions to $\phi$ are constant. In each of the depicted regions, the label $[\alpha,\beta]$ means that the contribution of the $0$-basepoint to $\phi$ is $\alpha\in Z(\A)$ and the $1$-basepoint contribution to $\phi$ is $\beta\in Z(\A)$. In particular, the value of $\phi$ is $\alpha+\beta$. 
For example, on the interior of the cone spanned by $(0,1,1)$, $(0,-1,1)$, $(1,-1,1)$, $(1,0,1)$, the value of $\phi$ is $\overline{5B}+\overline{3B}\in Z(\A)$.

As is predicted by Theorem \ref{thm:fan}, the lineality space of $\Sigma_\pi$ (that is, the direction in which $\phi$ is constant) is exactly the image of $\pi^*$:  the span of $(0,1,0)$. Projecting the regions of constancy of $\phi$ onto the first and third coordinates, we obtain the fan on the left of Figure \ref{fig:fans}.
The values of $\phi$ on the interior of these regions are exactly the generic limits of $C_{\pi,\ff}$. Each of these limits is a pair of torus-invariant lines meeting in a torus fixed point. The corresponding edges of $\A$ are colour-coded on the right of Figure \ref{fig:label}. The six uncoloured edges ($\overline{23}$, $\overline{56}$, $\overline{1A}$, $\overline{1B}$, $\overline{4A}$, $\overline{4B}$) are not $0$- or $1$-faces of $\A$ (cf.~Definition \ref{defn:iface}) and so do not arise in any generic limit. 
	We note that since all these limits are reduced conics, in this case the Hilbert-Chow morphism induces an isomorphism of normalizations of the orbit closures in the Hilbert scheme and Chow scheme. The orbit has the same dimension as the component $Z_\pi$ of the Hilbert scheme, so the normalization of $Z_\pi$ is the toric variety corresponding to the left fan of Figure \ref{fig:fans}. 

	We now consider the Cayley structure $\pi'$ instead. By Theorem \ref{thm:fan}, we know that the quasifan $\Sigma_{\pi'}$ will have lineality space spanned by $(0,1,1)$. A slice of this quasifan is depicted on the left of Figure \ref{fig:fanochow}. The maximal cones are labeled with the corresponding values of $\phi$; we notice that in this case, some of the generic limits are non-reduced. After an appropriate choice of coordinates, the projection of $\Sigma_{\pi'}$ to the quotient space is exactly the fan on the left of Figure \ref{fig:fans}.

	One may conduct a similar analysis for the other length $1$ Cayley structures of Example \ref{ex:fano}. The resulting fans are, up to change of coordinates, exactly those pictured in Figure \ref{fig:fans}.
\end{ex}
\begin{figure}
\begin{tikzpicture}[scale=1.5]
\draw[fill] (-1,0) circle [radius=0.04] node[left] {$4$};
\draw[fill] (0,-1) circle [radius=0.04] node[below right] {$6$};
\draw[fill] (-1,-1) circle [radius=0.04] node[below left] {$5$};
\draw[fill] (1,0) circle [radius=0.04] node[right] {$1$};
\draw[fill] (0,1) circle [radius=0.04] node[above left] {$3$};
\draw[fill] (1,1) circle [radius=0.04] node[above right] {$2$};
\draw[fill] (.5,1.3) circle [radius=0.04] node[above] {$A$};
\draw[fill] (-.5,-1.3) circle [radius=0.04] node[below] {$B$};
\draw (0,1) -- (-1,0) -- (-1,-1) -- (0,-1) -- (1,0) -- (1,1) -- (.5,1.3) -- (1,0) -- (.5,1.3) -- (0,-1) -- (.5,1.3) -- (-1,-1) -- (.5,1.3) -- (-1,0) -- (.5,1.3) -- (0,1);
\draw (-1,-1) -- (-.5,-1.3) -- (0,-1);
\draw[dashed] (0,1) -- (1,1) -- (-.5,-1.3);
\draw[dashed] (1,0) -- (-.5,-1.3);
\draw[dashed] (-1,0) -- (-.5,-1.3);
\draw[dashed] (0,1) -- (-.5,-1.3);
\end{tikzpicture}
\tiny{
\begin{tikzpicture}[scale=1.5]
\draw (0,1) -- (-1,0) -- (-1,-1) -- (0,-1) -- (1,0) -- (1,1) -- (.5,1.3) -- (1,0) -- (.5,1.3) -- (0,-1) -- (.5,1.3) -- (-1,-1) -- (.5,1.3) -- (-1,0) -- (.5,1.3) -- (0,1);
\draw (-1,-1) -- (-.5,-1.3) -- (0,-1);
\draw[dashed] (0,1) -- (1,1) -- (-.5,-1.3);
\draw[dashed] (1,0) -- (-.5,-1.3);
\draw[dashed] (-1,0) -- (-.5,-1.3);
\draw[dashed] (0,1) -- (-.5,-1.3);
\draw[thick,orange] (-1,-1) -- (.5,1.3) -- (0,1);
\draw[thick,yellow] (-1,-1) -- (-.5,-1.3) -- (0,1);
\draw[thick,green] (1,1) -- (.5,1.3) -- (0,-1);
\draw[thick,cyan] (1,1) -- (-.5,-1.3) -- (0,-1);
\draw[thick,violet] (1,1) -- (1,0) -- (0,-1);
\draw[thick,red] (-1,-1) -- (-1,0) -- (0,1);
\draw[draw=none] (-1.5,-1.6) -- (0,0) ;
\end{tikzpicture}
}\caption{Vertex labeling and generic limiting cycles (Example \ref{ex:fanochow})}\label{fig:label}
\end{figure}
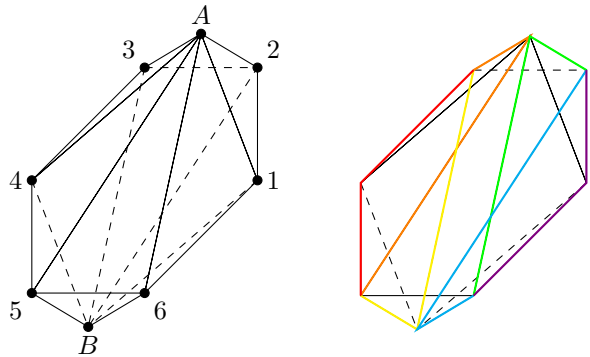
\begin{figure}
{\tiny{
\begin{tikzpicture}[scale=1.45]
\draw (0,2) -- (0,-2);
\draw (1,2) -- (1,-2);
\draw (-1,2) -- (-1,-2);
\draw (-2,0) -- (-1,0) -- (0,-1) -- (1,-1) -- (2,-2);
\draw (-2,2) -- (-1,1) -- (0,1) -- (1,0) -- (2,0);
\draw[fill] (1,0) circle [radius=0.04] node[above right] {$(1,0,1)$};
\draw[fill] (-1,0) circle [radius=0.04] node[above left] {$(-1,0,1)$};
\draw[fill] (0,1) circle [radius=0.04] node[right] {$(0,1,1)$};
\draw[fill] (-1,1) circle [radius=0.04] node[left] {$(-1,1,1)$};
\draw[fill] (0,-1) circle [radius=0.04] node[below right] {$(0,-1,1)$};
\draw[fill] (1,-1) circle [radius=0.04] node[right] {$(1,-1,1)$};
\draw[fill,blue] (-.5,1.5)  node {$[0,\overline{2B}+\overline{6B}]$};
\draw[fill,blue] (-.5,.5)  node {$[\overline{6B},\overline{2B}]$};
\draw[fill,blue] (-.5,-1.5)  node {$[\overline{2B}+\overline{6B},0]$};
\draw[fill,blue] (.5,1.5)  node {$[0,\overline{3B}+\overline{5B}]$};
\draw[fill,blue] (.5,-.5)  node {$[\overline{5B},\overline{3B}]$};
\draw[fill,blue] (.5,-1.5)  node {$[\overline{3B}+\overline{5B},0]$};
\draw[fill,blue] (-1.5,2.1)  node {$[0,\overline{12}+\overline{16}]$};
\draw[fill,blue] (-1.5,.5)  node {$[\overline{16},\overline{12}]$};
\draw[fill,blue] (-1.5,-1.5)  node {$[\overline{12}+\overline{16},0]$};
\draw[fill,blue] (1.5,.5)  node {$[0,\overline{34}+\overline{45}]$};
\draw[fill,blue] (1.5,-.5)  node {$[\overline{45},\overline{34}]$};
\draw[fill,blue] (1.5,-2.1)  node {$[\overline{34}+\overline{45},0]$};
\draw (0,-2.3) node {$z=1$ slice};
\end{tikzpicture} 
\begin{tikzpicture}[scale=1.45]
	\draw[draw=none](-.15,0)--(.15,0);
	\draw[dashed,thick,gray](0,2.1) -- (0,-2.4);
\end{tikzpicture} 
\begin{tikzpicture}[scale=1.45]
\draw (0,2) -- (0,-2);
\draw (1,2) -- (1,-2);
\draw (-1,2) -- (-1,-2);
\draw (-2,0) -- (-1,0) -- (0,-1) -- (1,-1) -- (2,-2);
\draw (-2,2) -- (-1,1) -- (0,1) -- (1,0) -- (2,0);
\draw[fill] (1,0) circle [radius=0.04] node[above right] {$(1,0,-1)$};
\draw[fill] (-1,0) circle [radius=0.04] node[above left] {$(-1,0,-1)$};
\draw[fill] (0,1) circle [radius=0.04] node[right] {$(0,1,-1)$};
\draw[fill] (-1,1) circle [radius=0.04] node[left] {$(-1,1,-1)$};
\draw[fill] (0,-1) circle [radius=0.04] node[below right] {$(0,-1,-1)$};
\draw[fill] (1,-1) circle [radius=0.04] node[right] {$(1,-1,-1)$};
\draw[fill,blue] (-.5,1.5)  node {$[0,\overline{2A}+\overline{6A}]$};
\draw[fill,blue] (-.5,.5)  node {$[\overline{6A},\overline{2A}]$};
\draw[fill,blue] (-.5,-1.5)  node {$[\overline{2A}+\overline{6A},0]$};
\draw[fill,blue] (.5,1.5)  node {$[0,\overline{3A}+\overline{5A}]$};
\draw[fill,blue] (.5,-.5)  node {$[\overline{5A},\overline{3A}]$};
\draw[fill,blue] (.5,-1.5)  node {$[\overline{3A}+\overline{5A},0]$};
\draw[fill,blue] (-1.5,2.1)  node {$[0,\overline{12}+\overline{16}]$};
\draw[fill,blue] (-1.5,.5)  node {$[\overline{16},\overline{12}]$};
\draw[fill,blue] (-1.5,-1.5)  node {$[\overline{12}+\overline{16},0]$};
\draw[fill,blue] (1.5,.5)  node {$[0,\overline{34}+\overline{45}]$};
\draw[fill,blue] (1.5,-.5)  node {$[\overline{45},\overline{34}]$};
\draw[fill,blue] (1.5,-2.1)  node {$[\overline{34}+\overline{45},0]$};
\draw (0,-2.3) node {$z=-1$ slice};
\end{tikzpicture}
}
}
\caption{Contributions to $\phi(v)$ for $\pi$ (Example \ref{ex:fanochow})}\label{fig:phiv1}
\end{figure}
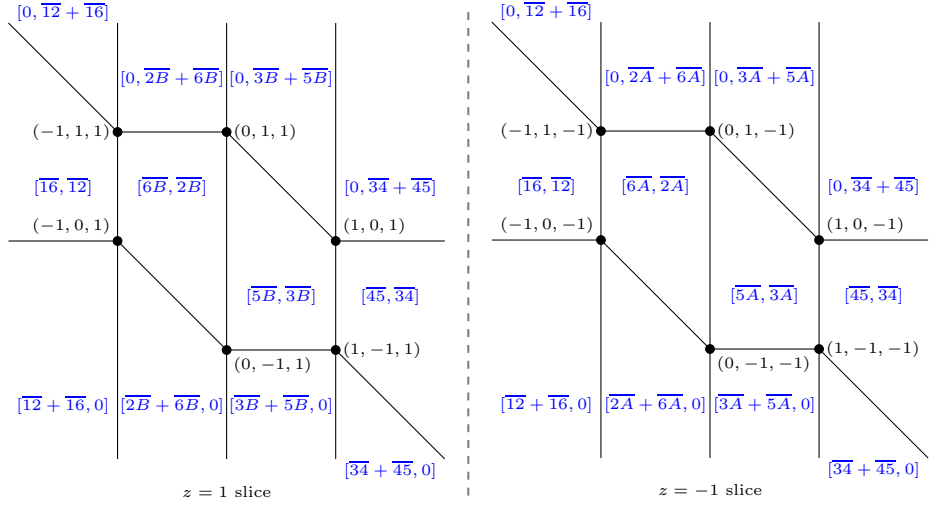
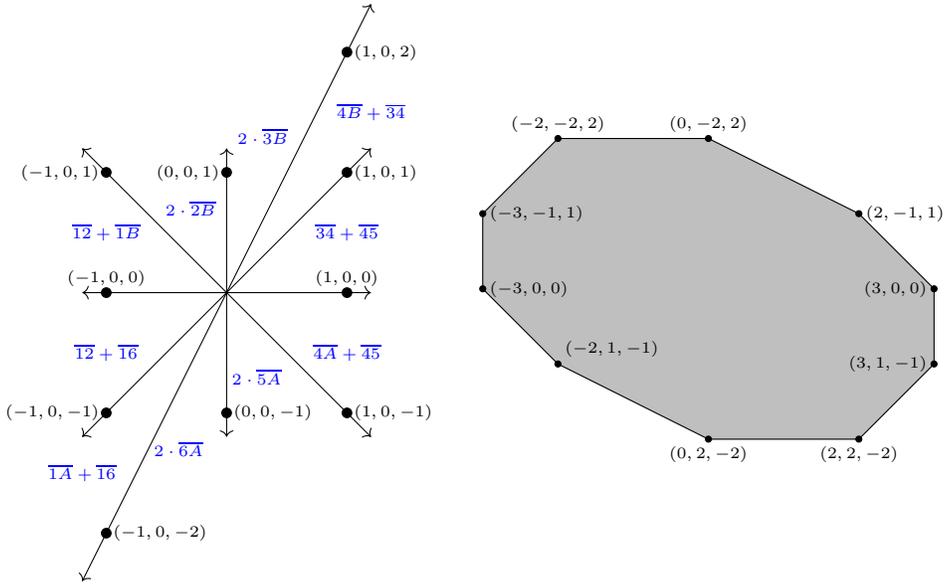
\begin{figure}
	\tiny\begin{tikzpicture}[scale=1.6]
		\draw[->] (0,0) -- (-1.2,-2.4);
		\draw[->] (0,0) -- (1.2,2.4);
		\draw[->] (0,0) -- (-1.2,0);
		\draw[->] (0,0) -- (1.2,0);
		\draw[->] (0,0) -- (-1.2,1.2);
\draw[->] (0,0) -- (0,1.2);
\draw[->] (0,0) -- (1.2,1.2);
\draw[->] (0,0) -- (-1.2,-1.2);
\draw[->] (0,0) -- (0,-1.2);
\draw[->] (0,0) -- (1.2,-1.2);
\draw[fill] (1,0) circle [radius=0.04] node[above] {$(1,0,0)$};
\draw[fill] (-1,0) circle [radius=0.04] node[above] {$(-1,0,0)$};
\draw[fill] (0,1) circle [radius=0.04] node[left] {$(0,0,1)$};
\draw[fill] (-1,1) circle [radius=0.04] node[left] {$(-1,0,1)$};
\draw[fill] (1,1) circle [radius=0.04] node[right] {$(1,0,1)$};
\draw[fill] (0,-1) circle [radius=0.04] node[right] {$(0,0,-1)$};
\draw[fill] (-1,-1) circle [radius=0.04] node[left] {$(-1,0,-1)$};
\draw[fill] (1,-1) circle [radius=0.04] node[right] {$(1,0,-1)$};
\draw[fill] (1,2) circle [radius=0.04] node[right] {$(1,0,2)$};
\draw[fill] (-1,-2) circle [radius=0.04] node[right] {$(-1,0,-2)$};
\draw[blue] (-.3,.7)  node {$2\cdot \overline{2B}$};
\draw[blue] (.3,1.3)  node {$2\cdot \overline{3B}$};
\draw[blue] (1.2,1.5)  node {$\overline{4B}+\overline{34}$};
\draw[blue] (1,.5)  node {$\overline{34}+\overline{45}$};
\draw[blue] (1,-.5)  node {$\overline{4A}+\overline{45}$};
\draw[blue] (.25,-.7)  node {$2\cdot\overline{5A}$};
\draw[blue] (-.4,-1.3)  node {$2\cdot\overline{6A}$};
\draw[blue] (-1.2,-1.5)  node {$\overline{1A}+\overline{16}$};
\draw[blue] (-1,-.5)  node {$\overline{12}+\overline{16}$};
\draw[blue] (-1,.5)  node {$\overline{12}+\overline{1B}$};
	\end{tikzpicture}
	\begin{tikzpicture}
		\draw[draw=none] (-3.5,-3.9) -- (0,0);
\draw[fill,lightgray] (2,-2)--(3,-1)--(3,0)--(2,1)--(0,2)--(-2,2)--(-3,1)--(-3,0)--(-2,-1)--(0,-2)--(2,-2);
\draw (2,-2)--(3,-1)--(3,0)--(2,1)--(0,2)--(-2,2)--(-3,1)--(-3,0)--(-2,-1)--(0,-2)--(2,-2);
\draw[fill] (2,-2) circle [radius=0.04] node[below] {$(2,2,-2)$};
\draw[fill] (0,-2) circle [radius=0.04] node[below] {$(0,2,-2)$};
\draw[fill] (3,-1) circle [radius=0.04] node[left] {$(3,1,-1)$};
\draw[fill] (3,0) circle [radius=0.04] node[left] {$(3,0,0)$};
\draw[fill] (2,1) circle [radius=0.04] node[right] {$(2,-1,1)$};
\draw[fill] (0,2) circle [radius=0.04] node[above] {$(0,-2,2)$};
\draw[fill] (-2,2) circle [radius=0.04] node[above] {$(-2,-2,2)$};
\draw[fill] (-3,1) circle [radius=0.04] node[right] {$(-3,-1,1)$};
\draw[fill] (-3,0) circle [radius=0.04] node[right] {$(-3,0,0)$};
\draw[fill] (-2,-1) circle [radius=0.04] node[above right] {$(-2,1,-1)$};
	\end{tikzpicture}

	\caption{$\Sigma_{\pi'}$ and $\Ch_T(C_{\pi',\ff})$ (Examples \ref{ex:fanochow} and \ref{ex:fanochow2})}\label{fig:fanochow}
	\end{figure}

	\subsection{Chow Polytopes}\label{sec:chow}
Let $Z$ be any be any cycle in $\PP^\ell$. Recall that its \emph{Chow polytope} $\Ch(Z)$ is the convex hull of the weights in $\ZZ^{\ell+1}$ in the corresponding Chow form \cite[\S6 Definition 3.1]{GKZ}. This is the moment polytope of the orbit closure of $\{Z\}$ in the Chow variety under the action by $(\KK^*)^{\ell+1}$; its normal fan is the fan corresponding to this toric variety. More generally, consider any subtorus $T'$ of $(\KK^*)^{\ell+1}$ with character lattice $M'$. The Chow polytope of $Z$ with respect to $T'$ is the linear projection $\Ch_{T'}(Z)$ of $\Ch(Z)$ induced by the map $\ZZ^{\ell+1}\to M'$, see e.g.~\cite[Case 4.8]{fink}. This is the moment polytope of the $T'$-orbit closure of $\{Z\}$.

Returning to our previous setup with a Cayley structure $\pi:\tau\to \Delta_{\ell}(d)$, we see from Theorem \ref{thm:fan} that the image of the fan $\Sigma_\pi$ is the normal fan to $\Ch_T(C_{\pi,\ff})$. We may explicitly describe this polytope using the combinatorics of $\pi$.
Let $\mu:Z(\tau)\to M$ be the linear map sending an edge   $\varepsilon \preceq\tau$ to $2\cdot L(\varepsilon)$ times the midpoint of $\varepsilon$.

\begin{cor}\label{cor:chow}
Assume that $\ff$ has distinct roots.	The Chow polytope $\Ch_T(C_{\pi,\ff})$ is the convex hull of the points
\[
\mu(\phi(v))\in M
\]
as $v$ ranges over general elements of each cone in $\Sigma_\pi$.
	\end{cor}
	
	Before proving the corollary, we recall several facts about Chow polytopes we will use;  these are stated in \cite[\S6.3]{GKZ} for the ``classical'' case $\Ch(Z)$ but are straightforward to generalize to the case of $\Ch_{T'}(Z)$. Firstly, $\Ch_{T'}(Z)$ is the convex hull of the polytopes $\Ch_{T'}(Z_v)$ where $v$ ranges over general one-parameter subgroups of $T'$ and $Z_v$ is the limit cycle of $Z$ under $v$. Secondly, for a cycle $Z=\sum c_iZ_i$, $\Ch_{T'}(Z)$ is the Minkowski sum
	\[
		\Ch_{T'}(Z)=\sum c_i\Ch_{T'}(Z_i).
	\]
Finally, we need the following:

\begin{lemma}\label{lemma:midpoint}
Let $\varepsilon$ be an edge of $\A$. Then $\Chow_T(X_\varepsilon)=\mu(\varepsilon)$.
\end{lemma}
\begin{proof}
	First consider the special case that $\A=\Delta_n(1)$. Then indeed $\Chow(X_\varepsilon)=\mu(\varepsilon)$, see \cite[\S 6 Example 3.4]{GKZ}. Returning to the situation of general $\A$, by taking $n=\#\varepsilon-1$, we may obtain an affine-linear bijection $\Delta_n(1)\to \varepsilon$ with some edge $\varepsilon'$ of $\Delta_n(1)$ mapping to the endpoints of $\varepsilon$. 

	There is an affine-linear Cayley structure $\varepsilon \to \Delta_{1}(L(\varepsilon))$; composition gives a Cayley structure $\pi':\Delta_n(1)\to \Delta_{1}(L(\varepsilon))$. The toric stratum $X_\varepsilon$ is just $C_{\pi'}$. We may apply Theorem \ref{thm:limit} to any $v$ minimized on $\varepsilon'$ to obtain $L(\varepsilon)\cdot \{X_{\varepsilon'}\}$ as a limit of $C_{\pi'}$. Since this is torus-fixed with respect to $(\KK^*)^{n+1}$, $L(\varepsilon)\cdot \mu(\varepsilon')$ is a vertex of $\Ch(X_\varepsilon)$. 

Since $X_\varepsilon$ is torus fixed with respect to $T$, $\Ch_T(X_\varepsilon)$ is a single vertex. It follows that this is simply the projection of $L(\varepsilon)\cdot \mu(\varepsilon')$ under the map induced by $\Delta_n(1)\to \varepsilon$. This map takes $\mu(\varepsilon')$ to the sum of the endpoints of $\varepsilon$, so $L(\varepsilon)\cdot \mu(\varepsilon')$ maps to $\mu(\varepsilon)$ as desired.
\end{proof}

The corollary is now immediate:
\begin{proof}[Proof of Corollary \ref{cor:chow}]
	This follows from the above discussion, Lemma \ref{lemma:limit}, and Lemma \ref{lemma:midpoint}.
\end{proof}

	\begin{rem}
		One could also prove Corollary \ref{cor:chow} and Theorem \ref{thm:fan} using tropical geometry. Indeed, since $\rho_{\pi,\ff}$ is the composition of a linear map with a monomial map, the tropicalization of $C_{\pi,\ff}$ (as a subvariety of the toric variety $X_\tau$) is given by the image of the standard tropical line in $\ell$-dimensional tropical projective space under the linear map $\pi^*\otimes \RR:\RR^{\ell+1}\to V$. One may compute the multiplicities of its cones using \cite[Theorem C.1]{OP}. An application of e.g.~\cite[Corollary 4.11]{fink} in this situation yields a description of the vertices of the Chow polytope. It is straightforward to see that this coincides with the result of Corollary \ref{cor:chow}; we leave this as an exercise to the tropically-minded reader.

		Although this approach to Corollary \ref{cor:chow} and Theorem \ref{thm:fan} is arguably simpler than the one we have taken here, we do not see how to obtain the full strength of Theorem \ref{thm:limit} (\emph{for non-general $v$}) using tropical methods.
	\end{rem}

	\begin{ex}\label{ex:fanochow2}
		Continuing Example \ref{ex:fanochow}, we use Corollary \ref{cor:chow} to compute $\Ch_T(C_{\pi',\ff})$, the Chow polytope for the Cayley structure $\pi'$. The result is pictured on the right of Figure \ref{fig:fanochow}. Note that the normal fan of this polytope is exactly the quasifan $\Sigma_{\pi'}$. 
	\end{ex}

	\subsection{Blowing Up}\label{sec:blowup}
In this subsection we will prove Theorem \ref{thm:limit}. Before doing so, we discuss the behaviour of families induced from Cayley structures under blowup.
Fix a Cayley structure $\pi:\tau\to\Delta_\ell(d)$.

Suppose that for some natural number $j$, we have the following data:
\begin{align*}
	a_i^{(j)},b_i^{(j)}\in\KK[z] \qquad i=0,\ldots,\ell
\end{align*}
and an affine linear map $\lambda^{(j)} : \tau \to \ZZ$.
Assume that for each $i$, $a_i^{(j)}$ and $b_i^{(j)}$ are not both divisible by $z$.
We then set
\[
	f_i^{(j)}=a_i^{(j)}y_0^{(j)}+b_i^{(j)}y_1^{(j)}\qquad \ff[j]=(f_0^{(j)},\ldots,f_\ell^{(j)}).
\]
We consider the rational map
\[
	\phi:\PP^1\times \Aff^1\dashrightarrow X_\tau
\]
where for $z\in \Aff^1$ and $y_0^{(j)},y_1^{(j)}$ homogeneous coordinates on $\PP^1$, we have 
\begin{align*}
	x_u=\begin{cases}
		0& u\notin \tau\\
		z^{\lambda^{(j)}(u)}\cdot {\ff[j]}^{\pi(u)} & u\in\tau.
	\end{cases}
\end{align*}

Fix $0 \leq k \leq \ell$ and assume that $a_k^{(j)},b_k^{(j)}$ are constants, $b_k^{(j)}\neq 0$, and $V(f_k^{(j)},z)$ is distinct from $V(f_i^{(j)},z)$ for $i\neq k$. Set $\gamma_j=a_k^{(j)}/b_k^{(j)}$.

We blow up $\PP^1\times\Aff^1$ at the point $V(\gamma_j y_0^{(j)}+ y_1^{(j)},z)=V(f_k^{(j)},z)$ and consider the induced map, which we again call $\phi$,
\[
\phi : \mathrm{Bl}_{V(f_k^{(j)},z)}(\PP^1 \times \Aff^1) \dashrightarrow X_\tau.
\]
We first consider the chart of the blowup with local coordinates
$z',y_0^{(j)},y_1^{(j)}$ where
\[z=z'(y_1^{(j)}/y_0^{(j)}+\gamma_{j}) = z' f_k^{(j)} / (b_k^{(j)} y_0^{(j)}).\]
In these coordinates, we have
\begin{align*}
	x_u=\begin{cases}
		0& u\notin \tau\\
		(z')^{\lambda^{(j)}(u)}(b_k^{(j)}y_0^{(j)})^{-\lambda^{(j)}(u)}{\ff[j]}^{\pi(u)+\lambda^{(j)}(u)\cdot e_k} & u\in\tau.
	\end{cases}
\end{align*}
\begin{lemma}\label{lemma:basepoint}
Suppose that there exists $w\in \tau$ that minimizes both $\lambda$ and $\lambda + \pi^*(e_k^*)$.
Then in the above coordinates, $\phi$ does not have a basepoint at $f_k^{(j)}=z'=0$.
\end{lemma}
\begin{proof}
	We show that for any $u\in \tau$, $x_u/x_w$ is regular at $f_k^{(j)}=z'=0$. Since $x_w/x_w=1$, this implies there is no basepoint.
	Fix any $u\in \tau$. Using the above description of $x_u$, we have
	\[x_u/x_w=(z')^{\lambda^{(j)}(u)-\lambda^{(j)}(w)}{\left(f_k^{(j)}\right)}^{e_k^*(\pi(u))+\lambda^{(j)}(u)-(e_k^*(\pi(w))+\lambda^{(j)}(w))}\cdot \zeta \]
	where $\zeta$ is regular at $f_k^{(j)}=z'=0$.
	But by assumption,
	\begin{align*}
\lambda^{(j)}(u)-\lambda^{(j)}(w)\geq 0\\
e_k^*(\pi(u))+\lambda^{j}(u)-(e_k^*(\pi(w))+\lambda^{j}(w))\geq 0
\end{align*}
so $x_u/x_w$ is regular.
\end{proof}

We now consider the other chart of the blowup with local coordinates $z$ and $y_0^{(j+1)},y_1^{(j+1)}$ with
\[y_1^{(j)}/y_0^{(j)}+\gamma_{j}=zy_1^{(j+1)}/y_0^{(j+1)}.\]
Let $E_{j+1}$ denote the closure of the line $z=0$ in this chart.
For $u\in \tau$, define
\[
	\lambda^{(j+1)}(u)=
	\lambda^{(j)}(u)+  e_k^*(\pi(u)).
\]
Likewise, we set
\begin{align*}
	a^{(j+1)}_i&=\begin{cases}
		0 & \qquad i=k\\
		a_i^{(j)}-b_i^{(j)}\gamma_{j} & \qquad\textrm{else}
	\end{cases}\\
	b_i^{(j+1)}&=\begin{cases}
		b_i^{(j)} & \qquad i=k\\
		zb_i^{(j)} & \qquad\textrm{else}\\
	\end{cases}
\end{align*}
and
\[
	f_i^{(j+1)}=a_i^{(j+1)}y_0^{(j+1)}+b_i^{(j_1)}y_1^{(j+1)}\qquad \ff[j+1]=(f_0^{(j+1)},\ldots,f_\ell^{(j+1)}).
\]
It is straightforward to check that in these coordinates, $\phi$ is given by
\begin{align*}
	x_u=\begin{cases}
		0& u\notin \tau\\
		z^{\lambda^{(j+1)}(u)}\cdot {\ff[j+1]}^{\pi(u)} & u\in\tau.
	\end{cases}
\end{align*}

\begin{proof}[Proof of Theorem \ref{thm:limit}]
Fix $\ff$ such that the entries have distinct roots. After an automorphism of $\PP^1$, we may assume that none of the $f_i$ have $V(y_0)$ as a root.
Let $a_i^{(0)},b_i^{(0)}\in\KK$ be such that $f_i^{(0)}=f_i$. Consider $v\in N_\tau$; after replacing $v$ by a sufficiently divisible multiple, we may assume that for every cone $\sigma\in \Sigma$ whose relative interior intersects $v+\RR_\geq \pi^*(e_i*)$, there is an integer $j\geq 0$ such that $v+ j\cdot \pi^*(e_i^*)$ is in the relative interior of $\sigma$.
Note that taking a multiple of $v$ does not change the limit cycle.
Let $\lambda^{(0)}:\tau\to \ZZ$ be any affine-linear map whose linear part is exactly $v$.
We claim that for any $j\geq 0$, and for any $0\leq k \leq \ell$, there exists $w\in \tau$ that minimizes both $\lambda+j\cdot e_k^*\circ \pi$ and $\lambda+(j+1)\cdot e_k^*\circ \pi$. Indeed, by our assumption on $v$, the linear parts of these two maps either belong to the same cone of $\Sigma$, or one belongs to a cone which is a face of the other. It follows that one of the corresponding faces of $\tau$ must be a subface of the other, hence such $w$ exists.

Considering the map $\phi$ as above, for $z\in \KK^*$
\[
	v(z) \cdot C_{\pi,\ff}=\phi(\PP^1,z)
\]
by construction.
The restriction of $\phi$ to  $E_0$, the line where $z$ vanishes, is given by $\rho_{\pi|_{\tau^v},\ff}$. The basepoints are given by $V(f_i)$ for those $i$ that are basepoints of $\pi|_{\tau^v}$. If $\pi|_{\tau^v}$ is constant, the line $z=0$ gets contracted to a point by $\phi$. Otherwise, assuming that $\ff$ is sufficiently general, the image of this line is $C_{\pi^v,\ff}$ by construction of $\pi^v$ and Theorem \ref{thm:primitive}. Moreover, by loc. cit. the multiplicity of the pushforward of the line $E_0$ is exactly $m^v$.

It remains to resolve the basepoints of $\phi$ and determine their contributions to the limiting cycle. Fix a basepoint $k$ of $\pi|_{\tau^v}$ and let $\gamma_0=a_k^{(0)}/b_k^{(0)}$. We blow up and use the coordinates discussed above; note that $a_k^{(0)}$ and $b_k^{(0)}$ are constants.

By Lemma \ref{lemma:basepoint} and the discussion above, $\phi$ has no basepoint at $E_0\cap E_1$.
We obtain
\[
	\lambda^{(1)}(u)=
	\lambda^{(0)}(u)+ e_k^*(\pi(u)).
\]
and
\begin{align*}
	a^{(1)}_i&=\begin{cases}
		0 & \qquad i=k\\
		a_i^{(0)}-b_i^{(0)}\gamma_{0} & \qquad\textrm{else}
	\end{cases}\\
	b_i^{(1)}&=\begin{cases}
		b_i^{(0)} & \qquad i=k\\
		zb_i^{(0)} & \qquad\textrm{else}\\
	\end{cases}.\\
\end{align*}

The function $f_k^{(1)}$ then vanishes at $(1:0)$; note that $a_k^{(1)}$ and $b_k^{(1)}$ are again constants. We continue in this manner: taking $\gamma_1=\ldots=\gamma_j=0$ and blowing up $j\geq 0$ more times leads to 
\[
	\lambda^{(j+1)}(u)=
	\lambda^{(0)}(u)+ (j+1)\cdot e_k^*(\pi(u)).
\]
and
\begin{align*}
	a^{(j+1)}_i&=\begin{cases}
		0 & \qquad i=k\\
		a_i^{(0)}-b_i^{(0)}\gamma_{0} & \qquad\textrm{else}
	\end{cases}\\
	b_i^{(j+1)}&=\begin{cases}
		b_i^{(0)} & \qquad i=k\\
		z^{j+1}b_i^{(0)} & \qquad\textrm{else}\\
	\end{cases}.\\
\end{align*}
The face $\tau'$ of $\tau$ on which $\lambda^{(j+1)}$ is minimal is exactly $\tau^{(v+(j+1)\pi^*(e_k^*))}$.
Furthermore, restricting to the line $E_{j+1}$ on this chart, we obtain
\[	(f_i^{(j+1)})|_{E_{j+1}}=\begin{cases}
		b_i^{(0)}y_1 & \qquad i=k\\
		(a_i^{(0)}-b_i^{(0)}\gamma_{0})y_0 & \qquad\textrm{else}\\
	\end{cases}
\]
and the map $\phi|_{E_{j+1}}$ is given by 
\begin{align*}
	x_u=\begin{cases}
		0& u\notin \tau'\\
		{\ff[j+1]}^{\pi(u)}|_{E_{j+1}} & u\in\tau'.
	\end{cases}
\end{align*}
Again by Lemma \ref{lemma:basepoint}, the only possible basepoint of $\phi$ on $E_{j+1}$ is at $(1:0)$. Hence, by continuing this blowup procedure, we eventually resolve the basepoints coming from $f_k$.

The map $\phi|_{E_{j+1}}$ is non-constant if and only if $\tau'$ is a $k$-face. Hence, the map becomes non-constant exactly when choosing $j\geq 0$ such that $v+(j+1)\pi^*(e_k^*)\in\Sigma_k$, and in this case the cone corresponding to $\tau'$ is in $\Sigma_k^v$.

For such $j\geq 0$, set $f_i'=(f_i^{(j+1)})|_{E_{j+1}}$. All the $f_i'$ have the same roots except for $f_k'$, so the map 
\[
	\phi|_{E_{j+1}}=t_k'\cdot \rho_{\kappa_i\circ \pi|_{\tau'}}=t_k'\cdot \rho_{\Res(\kappa_k\circ \pi|_{\tau'})}
\]
where, thinking of $V$ as $\Hom(\KK^*, T_\tau)$,
\[
t_k' = \pi^*(e_k^*)(b_k^{(0)}) \cdot \prod_{i \ne k} \pi^*(e_i^*)(a_i^{(0)} - b_i^{(0)}\gamma_0).
\]
Since \[
	f_i\wedge f_k=b_k^{(0)}(a_i^{(0)}-b_i^{(0)}/\gamma_0)
\]
we may act on $\PP^1$ by replacing $y_0^{(j+1)}$ with $b_k^{(0)}y_0^{(j+1)}$ and $y_1^{(j+1)}$ with $y_1^{(j+1)}/b_k^{(0)}$ to see that after rescaling coordinates, 
we may replace $t_k'$ with $t_k$ as defined prior to the statement of the theorem.
Finally, by Theorem \ref{thm:primitive}, the pushforward of $E_{j+1}$ under $\rho_{\Res(\kappa_k\circ \pi|_{\tau'})}$ is exactly $m_k^{\tau'} \cdot C_{\pi_k^{\tau'}}$.

Resolving each basepoint of $\phi$ in this manner, we obtain exactly the formula from the statement of the theorem. We note that in order to apply Theorem \ref{thm:primitive} for the Cayley structure $\pi^v$, we need to assume that $\ff$ is general. However, to apply it for the Cayley structures obtained after blowing up, we need no such assumption since these are Cayley structures of length one (whose input forms are thus automatically general).
\end{proof}

\subsection{Orbits in the Hilbert Scheme}
To probe the boundary of the Hilbert scheme $\Hilb_{dm+1}(X_\A)$ it is also of interest to ask for a description of the orbit closure of $[C_{\pi,\ff}]$ there. This seems challenging in general, but we do know the following:
\begin{prop}\label{prop:hilb}
	Let $\pi$ be a primitive smooth Cayley structure of degree $d$ and assume that $\ff$ is general. Then the fan describing the normalization of the orbit closure of $[C_{\pi,\ff}]$ in $\Hilb_{dm+1}(X_\A)$ is a refinement of the fan from Theorem \ref{thm:fan}.
\end{prop}
\begin{proof}
Let $Z'$ be the normalization of $\overline {T\cdot [C_{\pi,\ff}]}$ and $Z$ the normalization of $\overline {T\cdot \{C_{\pi,\ff}\}}$. Here, closures are being taken respectively in the Hilbert scheme and Chow variety.
	There is a natural $T$-equivariant morphism $\Hilb_{dm+1}(X_\A) \to \Chow_d(X_\A)$ taking a scheme to its underlying cycle; this map induces an isomorphism from the $T$-orbit of $[C_{\pi,\ff}]$ to the $T$-orbit of $\{C_{\pi,\ff}\}$ and hence a birational torus equivariant morphism $Z'\to Z$.
	
	This means that the fan for  $Z'$ is a subdivision of the fan for $Z$, see e.g.~\cite[\S3.3]{cls}. The fan for $Z$ is described in Theorem \ref{thm:fan}.
\end{proof}

In the case of conics, that is, degree two Cayley structures, we can give a more explicit answer.
It is well-known that any subscheme of $\PP^n$ with Hilbert polynomial $2m+1$ is a plane conic. For the sake of the reader, we provide a short proof of this fact:
\begin{lemma}\label{lemma:conic2}
	Any subscheme of $\PP^n$ with Hilbert polynomial $2m+1$ is contained in a plane.
\end{lemma}
\begin{proof}
	Let $Y\subset \PP^n$ be any plane conic. Then $Y$ is a complete intersection of a conic with $n-2$ linear forms. It is straightforward to check that the Piene-Schlessinger comparison theorem applies to $Y$ \cite{piene}. Since $Y$ is a complete intersection, it follows that $\Hilb_{2m+1}(\PP^n)$ is smooth at $[Y]$. By connectedness of the Hilbert scheme \cite{connectedness}, we conclude that $\Hilb_{2m+1}(\PP^n)$ is irreducible and only consists of plane conics.
\end{proof}

\begin{defn}Let $\pi:\tau\to \Delta_\ell(2)$ be a degree two primitive Cayley structure.
We define $\Sigma_\pi'$ to be the normal fan in $V=\Hom(\langle \tau\rangle ,\RR)$ of the convex hull of
\[\{u+v+w\in M\ |\ u,v,w\in \tau,\ \pi(u),\pi(v),\pi(w)\ \textrm{distinct}\}.
\]
\end{defn}

\begin{thm}\label{thm:conics}
	Let $\pi$ be a primitive Cayley structure of degree two and $\ff$ sufficiently general. Let $Z'$ be the normalization of 
	\[
		\overline{T\cdot [C_{\pi,\ff}]}\subseteq \Hilb_{2m+1}(X_\A).
	\]
	Then the fan describing $Z'$ is the image in $N_{\pi,\ff}\otimes \RR$ of the coarsest common refinement of $\Sigma_\pi$ and $\Sigma'_\pi$.
\end{thm}

\begin{proof}
	Let $\Lambda$ be the unique plane containing $C_{\pi,\ff}$. We will see that $\Sigma_\pi'$ is the fan corresponding to the closure in $\Gr(3,\#\A)$ of the torus orbit of $\Lambda$ under the action by $T$. Since any (possibly degenerate) conic is determined by the underlying cycle and the plane containing it (cf. Lemma \ref{lemma:conic2}), the claim of the theorem follows from Theorem \ref{thm:fan}.

	Let $\M(\Lambda)$ be the matroid of $\Lambda \subseteq \PP^{\#\A-1}$. The elements of the ground set are labeled by $u\in \A$, and a collection of elements $S\subset \A$ is a basis of $\M(\Lambda)$ if the projection of $\Lambda$ to the projective space with coordinates indexed by $S$ is bijective.
	It follows that any basis has exactly three elements, and by construction of $C_{\pi,\ff}$, these elements must lie in different fibers of $\pi$.

	In fact, this condition also suffices for three elements to form a basis. Indeed, consider the composition of the embedding $\PP^1\to \PP^{\ell}$ determined by $\ff$ with the second Veronese map $\PP^{\ell}\to \PP^{{{\ell+1}\choose 2}-1}$. The image of this map is a conic $C$; let $\Lambda'$ denote the plane containing it. Since $\ff$ is general, the matroid of $\Lambda'$ is the uniform rank 3 matroid on ${\ell+1}\choose 2$ elements. By construction of $C_{\pi,\ff}$, $u,v,w\in \A$ form a basis for $\M(\Lambda)$ if and only if $\pi(u),\pi(v),\pi(w)$ form a basis for $\M(\Lambda')$. Since the latter matroid is uniform, we see that $u,v,w$ form a basis if and only if they lie in different fibers of $\pi$.

	The \emph{matroid polytope} for $\M(\Lambda)$ is thus
	\[P=\conv\{e_u+e_v+e_w\in \ZZ^{\#\A}\ |\ u,v,w\in \tau,\ \pi(u),\pi(v),\pi(w)\ \textrm{distinct}\},
\]
see \cite{matroid}. The closure of the orbit of $\Lambda$ under the big torus ${(\KK^*)}^{\#\A}$ is the toric variety associated to the polytope $P$. Since we are interested instead in the $T$-orbit, we consider the projection of $P$ to $M$ under the map $e_u\mapsto u$. The resulting polytope describes the $T$-orbit closure of $\Lambda$, and its normal fan $\Sigma_\pi'$ the normalization thereof.
\end{proof}

\begin{ex}[General conics in $\PP^3$]\label{ex:P3}
	Let $\A=\Delta_{3}(1)$. The toric variety $X_\A$ is simply $\PP^3$.
	We consider the length seven Cayley structure $\pi:\A\to \Delta_7(2)$ sending $e_i$ to $e_i+e_{i+4}$ for $i=0,1,2,3$. This Cayley structure gives the most general conics in $\PP^3$. We will compare torus orbit closures in the Chow variety and Hilbert scheme.

	The normal fan of $\A$ in $N_\A$ has rays generated by the images of $e_0^*,\ldots,e_3^*$ in $N_\A=(\ZZ^4)^*/\langle e_0^*+\ldots +e_3^*\rangle$.
	The maximal cones of the normal fan are generated by any three of these rays. 

	One may calculate that the fan $\Sigma_\pi$ has rays generated by $\pm e_i^*$, $i=0,\ldots,3$. The six maximal cones are generated by collections of rays of the form $e_i^*,e_j^*,-e_k^*,-e_\ell^*$ for $\{i,j,k,\ell\}=\{0,1,2,3\}$. See Figure \ref{fig:P3} for a schematic representation of this fan: two ray generators are joined by a dashed black line segment if and only if they generate a face of a cone in $\Sigma_\pi$.
	This fan describes the normalization $Z$ of the orbit closure in $\Chow_2(\PP^3)$ by Theorem \ref{thm:fan}. This toric variety has six isolated singularities, each of which is a cone over the quadric surface.

	On the other hand, the fan $\Sigma_\pi'$ has rays generated by $-e_i^*$, $i=0,\ldots, 3$. The four maximal cones are generated by any set of three rays. In the schematic representation of Figure \ref{fig:P3}, two ray generators are joined by a gray line segment if and only if they generate a face of a cone in $\Sigma_\pi'$.
	The coarsest common refinement of $\Sigma_\pi$ and $\Sigma_\pi'$ thus has twelve maximal cones, generated by $e_i^*,-e_j^*,-e_k^*$ for $i,j,k$ distinct. Each of the six cones of $\Sigma_\pi$ is subdivided into two maximal cones, see Figure \ref{fig:P3}, taking both the dashed and gray line segments into account.
	By Theorem \ref{thm:conics} this fan describes the normalization  $Z'$ of the orbit closure in the Hilbert scheme.

	Geometrically, in this example the map $Z'\to Z$ is a crepant resolution of the singularities of $Z$.
\end{ex}

\begin{figure}
	\begin{tikzpicture}[scale=1.5]
		\draw[lightgray] (0,0) -- (0,2) -- (-1.73,1) -- (0,0);
		\draw[lightgray] (0,0) -- (0,2) -- (1.73,1) -- (0,0);
\draw[dashed] (0,0) -- (.58,1) -- (0,2); 
\draw[dashed] (0,0) -- (-.58,1) -- (0,2); 
\draw[dashed] (-1.73,1) -- (-.58,1); 
\draw[dashed] (1.73,1) -- (.58,1); 
\draw[fill] (0,0) circle [radius=0.04] node[below] {$-e_k^*$};
\draw[fill] (0,2) circle [radius=0.04] node[above] {$-e_i^*$};
\draw[fill] (1.73,1) circle [radius=0.04] node[right] {$-e_\ell^*$};
\draw[fill] (-1.73,1) circle [radius=0.04] node[left] {$-e_j^*$};
\draw[fill] (.58,1) circle [radius=0.04] node[below right] {$e_j^*$};
\draw[fill] (-.58,1) circle [radius=0.04] node[below left] {$e_\ell^*$};
\end{tikzpicture}
	\caption{Schematic representation of fans from Example \ref{ex:P3}. Black dashed lines are $2$-dimensional cones in $\Sigma_\pi$; gray lines are $2$-dimensional cones in $\Sigma'_\pi$.}\label{fig:P3}
\end{figure}
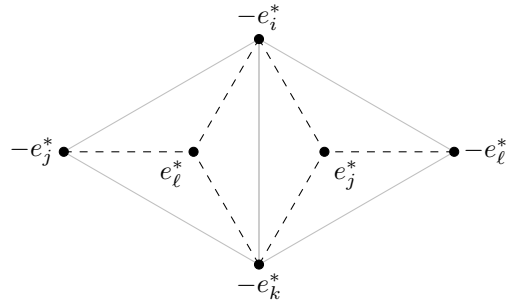

\begin{ex}\label{ex:fanohilb2}
	We continue our analysis of Example \ref{ex:cayley} by applying Theorem \ref{thm:conics}. For the Cayley structure $\pi$, we had already noted in Example \ref{ex:fanochow} that the fan $\Sigma_{\pi}$ describes the orbit closure in the Hilbert scheme. We can now see this in another way: the polytope used to construct the fan  $\Sigma_{\pi'}$ is depicted on the left of Figure \ref{fig:fanohilb2}. Its normal fan is exactly the quasifan $\Sigma_{\pi}$, so by Theorem \ref{thm:conics} we recover that the orbit closures in the Hilbert scheme and in the Chow variety coincide.

	We may instead consider the Cayley structure $\pi'$. The polytope used to construct the fan $\Sigma_{\pi'}'$ is depicted on the right of Figure \ref{fig:fanohilb2}. Its normal fan is a coarsening of the quasifan $\Sigma_{\pi'}$, so again by Theorem \ref{thm:conics} we recover that the orbit closures in the Hilbert scheme and in the Chow variety coincide.
	A similar analysis shows that the same statement holds for any of the length one Cayley structures considered in Example \ref{ex:fano}. Since the orbit closures have dimension two, and the corresponding Hilbert scheme components also have dimension two, we obtain that the normalizations of the Hilbert scheme components are the toric surfaces corresponding to the fans of Figure \ref{fig:fans}.
\end{ex}
\begin{figure}\tiny	
	\begin{tikzpicture}[scale=1.2]
		\draw[fill,lightgray] (-2,0) -- (-1,-1) -- (1,-1) -- (2,0) -- (1,1) -- (-1,1) -- (-2,0);
		\draw (-2,0) -- (-1,-1) -- (1,-1) -- (2,0) -- (1,1) -- (-1,1) -- (-2,0);
\draw[fill] (2,0) circle [radius=0.04] node[left] {$(2,0,0)$};
\draw[fill] (-2,0) circle [radius=0.04] node[right] {$(-2,0,0)$};
\draw[fill] (-1,-1) circle [radius=0.04] node[below] {$(-1,0,-1)$};
\draw[fill] (-1,1) circle [radius=0.04] node[above] {$(-1,0,1)$};
\draw[fill] (1,-1) circle [radius=0.04] node[below] {$(1,0,-1)$};
\draw[fill] (1,1) circle [radius=0.04] node[above] {$(1,0,1)$};
	\end{tikzpicture}
	\begin{tikzpicture}[scale=1.2]
		\draw[draw=none] (-3,0) -- (0,0);
		\draw[fill,lightgray] (-2,0) -- (-1,-1) -- (2,-1) -- (2,0) -- (1,1) -- (-2,1) -- (-2,0);
\draw (-2,0) -- (-1,-1) -- (2,-1) -- (2,0) -- (1,1) -- (-2,1) -- (-2,0);
\draw[fill] (2,0) circle [radius=0.04] node[left] {$(2,0,0)$};
\draw[fill] (-2,0) circle [radius=0.04] node[right] {$(-2,0,0)$};
\draw[fill] (-1,-1) circle [radius=0.04] node[below] {$(-1,1,-1)$};
\draw[fill] (-2,1) circle [radius=0.04] node[above] {$(-1,-1,1)$};
\draw[fill] (2,-1) circle [radius=0.04] node[below] {$(1,1,-1)$};
\draw[fill] (1,1) circle [radius=0.04] node[above] {$(1,-1,1)$};
	\end{tikzpicture}
	\caption{Matroid polytopes for $\pi$ and $\pi'$ (Example \ref{ex:fanohilb2})}\label{fig:fanohilb2}
\end{figure}
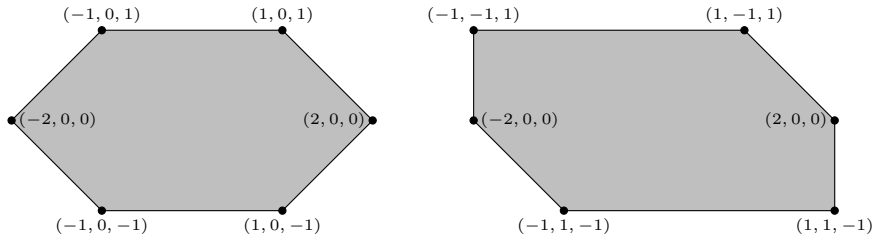
\bibliographystyle{amsalpha}
    \bibliography{paper}
    \end{document}